\documentclass[reqno,11pt,oneside]{amsart}
\usepackage{hyperref}
\usepackage{amssymb}

\usepackage{geometry}  
\usepackage[english]{babel}
\usepackage{epsfig}
\usepackage{hyperref}
\usepackage{float}
\usepackage[parfill]{parskip} 
\usepackage{todo}
\usepackage{color}

\usepackage{physics}
\usepackage{tensor} 
\usepackage{amsthm,amsfonts}
\usepackage{enumerate}
\usepackage{array}
\theoremstyle{plain}
\newtheorem{theorem}{Theorem}[section]

\theoremstyle{definition}
\newtheorem{definition}[theorem]{Definition}
\newtheorem{proposition}[theorem]{Proposition}
\newtheorem{remark}[theorem]{Remark}
\newtheorem{corollary}[theorem]{Corollary}
\newtheorem{lemma}[theorem]{Lemma}

\renewenvironment{proof}{{\noindent \bf  Proof.}}{\qed}

\def\flap{\Delta^{\frac{\alpha}{2}}}
\newcommand{\myF}[2]{F_{#2}\left[#1\right]}

\newcommand{\overbar}[1]{\mkern 1.5mu\overline{\mkern-1.5mu#1\mkern-1.5mu}\mkern 1.5mu}

\numberwithin{equation}{section}
\begin{document}
\title[\hfil space-time fractional EE\MakeLowercase{s}\hfil]{{\itshape Stochastic solutions for space-time fractional evolution equations on bounded domain}}
\author{L\MakeLowercase{orenzo} T\MakeLowercase{oniazzi}}
\date{\today}
\address{L. Toniazzi\newline
Department of Mathematics, University of Warwick, Coventry, United Kingdom. }
\email{l.toniazzi@warwick.ac.uk}

\thanks{The author is funded by the EPSRC, UK}
\maketitle
\begin{abstract} Space-time fractional evolution equations are a powerful tool to model diffusion displaying space-time heterogeneity. 
We prove existence, uniqueness and stochastic representation of classical solutions for an extension of Caputo evolution equations featuring time-nonlocal initial conditions. We discuss the interpretation of the new stochastic representation. As part of the proof a new result about inhomogeneous Caputo evolution equations is proven.
\end{abstract}
\section{Introduction}
It is a classical result that the solution to the standard heat equation $\partial_tu=\Delta u$, $u(0)=\phi_0$ allows the stochastic representation $u(t,x)=\mathbf E[\phi_0(X^{x,2}(t))]$, where $X^{x,2}$ is a Brownian motion started at $x\in\mathbb R^d$. 
Space-time fractional evolution equations (EEs) extend  the  heat equation by introducing space-time heterogeneity. This often is done by considering the Caputo EE $D^\beta_0u=-(-\Delta)^{\frac\alpha 2}u$, where one substitutes the local operators $\partial_t$ and $\Delta$ with fractional analogues. Respectively, the Caputo derivative $ D^\beta_0u(t) = c_\beta\int_0^tu'(r)(t-r)^{-\beta}\,dr $ and the fractional Laplacian $(-\Delta)^{\frac\alpha 2}u(x)= \mathcal F^{-1}(|\xi|^\alpha \mathcal F u(\xi))(x)$, where $\beta\in(0,1)$, $\alpha\in(0,2)$, $c_\beta=\Gamma(1-\beta)^{-1}$ and $\mathcal F$ is the Fourier transform (for standard references see \cite{kai,Bogdan}). It is well known that the fundamental solution to the Caputo EE is  the law of the non-Markovian anomalous diffusion $Y^x(t)=X^{x,\alpha}(\tau_0(t))$ (see, e.g., \cite{Meerschaert2012}). Here  $X^{x,\alpha}$  is  the rotationally symmetric $\alpha$-stable L\'evy process  started at $x\in\mathbb R^d$ and  $\tau_0(t)$ is the inverse process of the $\beta$-stable subordinator $X^\beta(t)$. The density of this beautiful formula was first observed in \cite{Zas97}. The time change interpretation first appeared in \cite{meershe,schef04}, based on \cite{BM01}. The process $Y^x$ displays space-heterogeneity due to the jump nature of $X^{x,\alpha}$. Also time-heterogeneity features in $Y^x$, as the time change $t\mapsto\tau_0(t)$ is constant precisely when the subordinator $t\mapsto X^\beta(t)$ jumps,  so that $t\mapsto Y^x(t)$ is trapped on such time intervals. This interesting trapping phenomenon leads to the process $Y^x$ spreading at a slower rate than $X^{x,\alpha}$. Indeed, in the physics literature the anomalous diffusion $Y^x$ is often referred to as a sub-diffusion when $\alpha=2$ (see, e.g., \cite{Zas94,Sai05,Mag07}). See \cite{schef04} for a characterisation of $Y^x$ as the scaling limit of continuous time random walks with heavy-tailed waiting times. See \cite{BC11} for a characterisation of $Y^x$ as the scaling limit of random conductance models or asymmetric Bouchaud's trap models ($\alpha=2$). See  \cite{Mag10,Mag15} for sample path properties of $Y^x$, and \cite{DS18,CK18} for heat kernel asymptotic formulas. Existence of classical solutions for Caputo EEs is generally a subtle problem. The works \cite{Ko03,BMN09,Caff16} tackle classical solutions on unbounded domains. Meanwhile the works  \cite{MeerChen,Vel09,Vel11,Leon} consider bounded domains, and all their proofs rely on the spectral decomposition of the spatial operator. Stochastic representations for solutions to time-nonlocal equations is an active area of theoretical research (see, e.g., \cite{MeerB,Chen17,HKT17,CK18}). Partly because they provide formulas in the general absence of closed forms along with suggesting probabilistic proof methods.  Moreover, such representations can be useful for particle tracking codes (see, e.g., \cite{ZMB08}). Let us remark that  Caputo EEs are applied in a variety of fields, such as physics, finance, economics, biology and hydrogeology (see, e.g., \cite{Zas97,Scal06,Scal00,MeeB13,Fed}).

In this work we focus on the following extension of the Caputo EE: the inhomogeneous space-time fractional EE on bounded domain with Dirichlet boundary conditions and time-nonlocal initial condition
\begin{equation}\label{preRL}
\left\{
\begin{split}
D_{\infty}^{\beta}  \tilde u(t,x)&=\flap_\Omega   \tilde u(t,x)+g(t,x), & \text{in }& (0,T]\times \Omega,\\ 
 \tilde u(t,x)&=0, &  \text{in }&[0,T]\times \partial\Omega, \\
 \tilde u(t,x)&=\phi(t,x),  &\text{in }&(-\infty,0]\times\Omega,
\end{split}
\right.
\end{equation}
where $\Omega\subset\mathbb R^d$ is a regular domain, $\flap_\Omega $ is the restricted fractional Laplacian\footnote{We define $\flap_\Omega$ on functions on $\Omega$, so that the Euclidean boundary $\partial\Omega$ makes sense in (\ref{preRL}). In the literature the operator $\flap_\Omega$ is often defined through the application of the singular integral definition of $-(-\Delta)^{\frac\alpha 2}$  to functions vanishing outside $\Omega$ (see, e.g., \cite{Bon16}).}, and the time operator $-D_{\infty}^{\beta}$  is the generator of the inverted $\beta$-stable subordinator\footnote{The operator  $D_{\infty}^{\beta}$ is often referred to as the Marchaud derivative in the fractional calculus literature (see, e.g., \cite{kilbas}).}
\begin{equation}
D_{\infty}^{\beta} f(t)=\int_0^{\infty}(f(t-r)-f(t))\,\frac{\Gamma(-\beta)^{-1}dr}{r^{1+\beta}},\quad t\in\mathbb R.
\label{gen}
\end{equation}
 As the main result of this work we prove existence and uniqueness of classical solutions to problem (\ref{preRL}) along with the  stochastic representation for the solution
\begin{equation}\label{SR}
\begin{split}
\tilde u(t,x)=&\ \mathbf E\left[\phi\left(-X^{t,\beta}(\tau_0(t)),X^{x,\alpha}(\tau_0(t))\right)\mathbf 1_{\{\tau_0(t)<\tau_\Omega(x)\}}\right] \\ 
&\ +\mathbf E \left[\int_0^{\tau_0(t)\wedge \tau_\Omega(x)}g \left(-X^{t,\beta}(s),X^{x,\alpha}(s)\right)ds\right],
\end{split}
\end{equation}
where the  processes $-X^{t,\beta}=t-X^\beta$ and  $X^{x,\alpha}$ are independent, and $\tau_{\Omega}(x)$ is the first exit time of $X^{x,\alpha}$ from $\Omega$. To see why problem (\ref{preRL}) extends the Caputo EE, let $\phi(t)=\phi(0)$ for every $t\in(-\infty,0)$ and $g=0$ in both (\ref{preRL}) and (\ref{SR}). Then 
\begin{equation*}
D_{\infty}^{\beta}  \tilde u(t)=\int_0^t (\tilde u(t-r)-\tilde u(t))\,\frac{\Gamma(-\beta)^{-1}dr}{r^{1+\beta}}-\frac{\phi(0)-\tilde u(t)}{\Gamma(1-\beta)}t^{-\beta}=D^\beta_0u(t),
\end{equation*}
where $u$ is the restriction of $\tilde u$ to $t\ge0$, and one obtains the homogeneous Caputo EE  and its solution, respectively. The recent works \cite{Du17,DYZ17} introduced a class of EEs that formally includes (\ref{preRL}). They are motivated by the success of related nonlocal EEs arising in image processing, peridynamics and heat conduction  (see, e.g.,  \cite{Gil08,Bob10,Sil10,Gu12}), and  the general lack of alternatives to Caputo-type time-nonlocal models. 
Part of their intent is to introduce initial conditions on the `past' ($\phi$ on $(-\infty,0)\times\Omega$).  Our stochastic solution  (\ref{SR}) appears to be new, and it provides an interesting interpretation for the time-nonlocal initial condition $\phi$. This is because the overshoot $W(t)=X^{t,\beta}(\tau_0(t))$ is the waiting/trapping time of the anomalous diffusion $X^{x,\alpha}(\tau_0(t))$. We  discuss an interpretation where the values of $\phi$ on $(-\infty,0)\times\Omega$ describe the initial condition at time $0$ with respect to the `depth' of $\Omega$, rather than the `past' of $\Omega$. To the best of our knowledge, there are no classical-wellposedness results for the EE (\ref{preRL}). Related weak-wellposedness results can be found in  \cite{Du17,DYZ17} (for certain general L\'evy kernels in (\ref{gen})) and indirectly in \cite{miao} (for abstract Markovian generators), meanwhile \cite{Allen17} considers uniqueness of weak solutions. Worth mentioning that our simple Lemma \ref{lem_equivsol} allows to obtain wellposedness and regularity results for EEs such as (\ref{preRL})  as corollaries of theorems concerning inhomogeneous Caputo EEs (see, e.g., \cite{Ko03,Caff16}). To see why the stochastic representation (\ref{SR}) is natural, one can formally apply the classical probabilistic intuition for elliptic boundary value problems (see, e.g., \cite[Introduction, \S 3]{dynkin1965}) to problem (\ref{preRL})  rewritten  as 
\begin{equation}\label{bvpp}
\left\{
\begin{split}
\mathcal L\tilde u&=-g, & \text{in }&\Gamma,\\
\tilde u&=\phi, &  \text{in }&\partial\Gamma,
\end{split}
\right.
\end{equation}
where $\mathcal L= (-D_{\infty}^{\beta}+\flap_\Omega)$ is the generator of the  process $ \{(-X^{t,\beta}(s),X^{x,\alpha}(s))\mathbf1_{\{s<\tau_\Omega(x)\}}\}_{s\ge0} $ taking values in $(-\infty,T]\times \Omega$, $ \Gamma=(0,T]\times \Omega$, and $\partial \Gamma:=(-\infty,0]\times\Omega\cup [0,T]\times\partial \Omega,$  with $\phi=0$ on $ (0,T]\times\partial \Omega$.

To prove  our main result, Theorem \ref{thm_main}, we derive  two results of independent interest. Namely: 
\begin{itemize}
	\item Theorem \ref{thm_sspostRL}: the stochastic representation 
	\begin{equation}\label{SRpostRL}\begin{split}
	u(t,x)=&\ \mathbf E\left[\phi_0\left(X^{x,\alpha}(\tau_0(t))\right)\mathbf 1_{\{\tau_0(t)<\tau_\Omega(x)\}}\right]\\
	&\ +\mathbf E \left[\int_0^{\tau_0(t)\wedge \tau_\Omega(x)}f \left(-X^{t,\beta}(s),X^{x,\alpha}(s)\right)ds\right],
	\end{split}
	\end{equation}
is the  unique classical solution to the inhomogeneous Caputo EE on bounded domain 
	\begin{equation}\label{postRL}
\left\{
\begin{split}
D_{0}^{\beta}   u(t,x)&=\flap_\Omega    u(t,x)+f(t,x), & \text{in }& (0,T]\times \Omega,\\ 
  u(t,x)&=0, &  \text{in }&[0,T]\times \partial\Omega, \\
  u(t,x)&=\phi_0(x),  &\text{in }&\{0\}\times\Omega;
\end{split}
\right.
\end{equation}
	\item Theorem \ref{thm_wspostRL}: the  stochastic representation (\ref{SRpostRL}) is a weak solution to problem (\ref{postRL}).
\end{itemize}
Let us outline our proof strategy for Theorem \ref{thm_main}. By plugging the values of $\phi$ in $\tilde u$, it is not hard to show the equivalence of classical solutions to problem (\ref{preRL}) and to problem (\ref{postRL}) with forcing term $f=g-D^\beta_\infty\phi$ and initial condition $\phi_0=\phi(0)$ (see Lemma \ref{lem_equivsol}).  Moreover,  a Dynkin formula argument proves that the respective stochastic representations (\ref{SR}) and  (\ref{SRpostRL}) agree (see Lemma \ref{lem_SRs}). Hence, it is enough to prove Theorem \ref{thm_sspostRL}. We do so by proving Theorem \ref{thm_wspostRL} and then showing the required regularity of the candidate solution (\ref{SRpostRL}). The main feature of our regularity assumption on the data $\phi$ and $g$ is the differentiability in time. This is a consequence of the regularity assumption  on $f$ in Theorem \ref{thm_sspostRL}, which we discuss now. Theorem \ref{thm_sspostRL} extends the proof of \cite[Theorem 5.1]{MeerChen}, where problem (\ref{postRL}) is treated for $f=0$. This proof uses separation of variables combing eigenfunction expansions of $\flap_\Omega$ with Mittag-Leffler solutions to the Caputo initial value problem. Our separation of variables formula for the second term in (\ref{SRpostRL}) reads
\begin{equation*}
\sum_{n=1}^\infty\psi_n(x) u_n(t) = \sum_{n=1}^\infty\psi_n(x) \int_0^t \langle f(s),\psi_n\rangle (t-s)^{\beta-1} \beta E_\beta'(-\lambda_n (t-s)^\beta)\,ds,
\label{eigenex}
\end{equation*}
 where  $E_\beta(t)=\sum_{k=0}^\infty t^k\Gamma(k\beta+1)^{-1}$ is a Mittag-Leffler function,  $\{\lambda_n,\psi_n\}_{n\in\mathbb N}$ is the system of eigenvalues-eigenfunctions of $\flap_\Omega$ and  $\langle \cdot,\cdot\rangle$ is the inner product on $\Omega$. Unsurprisingly, each $u_n$ is the solution  to the inhomogeneous Caputo initial value problem $D^{\beta}_0 u_n(t)=-\lambda_nu_n(t)+\langle f(t),\psi_n\rangle$, $u_n(0)=0$ (see \cite[Theorem 7.2]{kai}). As  we require differentiability of $t\mapsto u(t)$, we want to differentiate each $t\mapsto u _n(t)$. To compensate for the singularity of the Mittag-Leffler kernel $t^{\beta-1}E_\beta'(-\lambda_n t^\beta)$ we require differentiability of $t\mapsto f(t)$. Note that for the space fractional heat equation ($\beta=1$) the Mittag-Leffler kernel  is an exponential, and so continuity of $f$ is enough to differentiate the $u_n$'s. Related results in the literature also require differentiability on $f$ (see, e.g, \cite[Theorem 7.3]{Caff16}). Briefly, the arguments for Theorem \ref{thm_wspostRL} reduce the Caputo EE (\ref{postRL}) to a Poisson equation with zero boundary conditions on $\{0\}\times\Omega\cup [0,T]\times\partial\Omega$ by constructing space-time sub-Feller semigroups. We rely on the fact that the  generator $-D_{0}^{\beta} $ only requires boundary conditions on the trivial set $\{0\}$. These arguments are an extension of the ideas in \cite{HKT17}, and they appear versatile. For example, they can be used to prove stochastic weak solutions for problem (\ref{preRL}) with general nonlocal operators  in both space and time (ongoing work with the authors in \cite{DYZ17}). As far as we know,  stochastic representations for solutions such as (\ref{SRpostRL}) for time-nonolocal  EEs appear in  \cite{HKT17}, meanwhile in \cite{Meer05} the solution is given a representation via the superposition principle. Possibly worth mentioning that we do not invoke \cite[Theorem 3.1]{BM01} and all our methods work for the standard Laplacian case $\alpha=2$. 

This work is structured as follows: in Section \ref{sec_preliminaries} we provide general notation and basic results about several stochastic processes obtained from $-X^{t,\beta}$ and $X^{x,\alpha}$, with a focus on semigroup results. In Section \ref{sec_ws} we prove Theorem \ref{thm_wspostRL}. In Section \ref{sec_existence} we prove Theorem \ref{thm_sspostRL}. In Section \ref{sec_sspre} we prove that the stochastic representation (\ref{SR}) is the unique classical solution to the EE (\ref{preRL}). In Section \ref{sec_int} we discuss an interpretation of the stochastic representation (\ref{SR}).
\section{Preliminaries}\label{sec_preliminaries}
\subsection{General notation} 
 We denote by $\mathbb N,$ $ \mathbb R^d,$ $ \Gamma(\cdot),$ $ \mathbf 1_{A}(\cdot)$, $a\wedge b$, a.e., lhs and rhs, the set of natural numbers, the $d$-dimensional Euclidean space, the gamma function, the indicator function of the set $A$, the minimum between $a,b\in\mathbb R$, the statements almost everywhere with respect to Lebesgue measure, left hand side and right hand side, respectively. We define  the one parameter Mittag-Leffler function for $\beta\in(0,1)$ as $E_\beta(t)=\sum_{k=0}^\infty t^k\Gamma(k\beta+1)^{-1}$, $t\ge0$. We define the Banach spaces 
\begin{align*} 
B(A)=&\ \{f:A\to\mathbb R\text{ is bounded and measurable}\},\\
C(K)=&\ \{f\in B(K): f\text{ is continuous}\},\\
C_{\partial\Omega}(\Omega)=&\ \{f\in C(\overbar \Omega): f=0\text{ on }\partial \Omega\},\\
C_{0}([0,T])=&\ \{f\in C([0,T]): f(0)=0\}, \\
C_{\infty}((-\infty, T])=&\ \{f\in B((-\infty, T]): f\text{ is continuous and vanishes at infinity}\}, \\
C_{\partial\Omega}([0,T]\times \Omega)=&\ \{f \in C([0,T]\times \overbar\Omega): f=0\text{ on }\partial\Omega\},\\
C_{0,\partial\Omega}([0,T]\times \Omega)=&\ \{f \in C_{\partial\Omega}([0,T]\times \Omega): f(0)=0\},\\
C_{\infty,\partial\Omega}((-\infty, T]\times\Omega)=&\ \{f\in B((-\infty, T]\times\Omega):f\text{ is continuous and vanishes at infinity}  \},\\
C_{b,\partial\Omega}((-\infty, T]\times\Omega)=&\ \{f\in B((-\infty, T]\times\overbar\Omega):f\text{ is continuous and } f=0\text{ on }\partial\Omega   \},
\end{align*}
all equipped with the supremum norm, where $A$ is any subset of $\mathbb R^d$, the set $K\subset \mathbb R^d$ is compact, the set $\Omega\subset \mathbb R^d$ is bounded and open,  $ T\ge0$. For a function $f:A\to \mathbb R$ we denote its supremum norm by either $\|f\|_\infty$ or $\|f\|_{C(A)}$. We define the spaces
\begin{align*} 
C(O)=&\ \{f:O\to \mathbb R \text{ is continuous}\},\\
C^{k}(\Omega)=&\ \{f\in C(\Omega):f\text{ is $k$-times continuously differentiable} \}, \\
C^{k}_c(\Omega)=&\ \{f\in C(\Omega):f\in C^{k}(\Omega)\text{ and compactly supported} \}, \\
C^{\infty}_c(\Omega)=&\ \{f\in C(\Omega):f\text{ is smooth and compactly supported} \}, \\
C^1([0,T])=&\ \{f,f'\in C([0,T])\},\\
C^1_0([0,T])=&\ \{f,f'\in C_{0}([0,T])\},\\
C^{1}_\infty((-\infty, T])=&\ \{f,f'\in C_\infty((-\infty, T])\},\\
C^{1,k}((0,T)\times\Omega)=&\ \{f\in C((0,T)\times\Omega): f\text{ is  $1$-time and $k$-times continuously  }\\
& \quad\quad\text{differentiable in time and space, respectively}\},\\
C^{1,k}_c((0,T)\times\Omega)=&\ \{f\in C^{1,k}((0,T)\times\Omega): f\text{ is compactly supported}\},\\
C_{\partial\Omega}^1([0,T]\times \Omega)=& \ \{f\in C_{\partial\Omega}([0,T]\times \Omega):f\in C^{1,0}((0,T)\times\Omega), f'\in C_{\partial\Omega}([0,T]\times \Omega) \},\\
C_{\infty,\partial\Omega}^{n,k}((-\infty,T]\times \Omega)=& \ \{f\in C_{\infty,\partial\Omega}((-\infty,T]\times \Omega):\text{all derivatives up to order $n$ in time}\\
& \quad\quad\text{and $k$ in space exist and belong to }C_{\infty,\partial\Omega}((-\infty,T]\times \Omega) \},
\end{align*}
where the set $O\subset \mathbb R^d$ is open. We write $C_{\infty,\partial\Omega}^{1,0}((-\infty,T]\times \Omega)=C_{\infty,\partial\Omega}^1((-\infty,T]\times \Omega)$ and $C_{b,\partial\Omega}^1((-\infty, T]\times\Omega)= \{f,\partial_tf\in C_{b,\partial\Omega}((-\infty, T]\times\Omega)\}$. By $(L^1(O),\|\cdot\|_{L^1(O)})$, $(L^2(O),\|\cdot\|_{L^2(O)})$ and $(L^\infty(O),\|\cdot\|_{L^{\infty}(O)})$ we mean the standard  Banach spaces of real-valued Lebesgue integrable, square-integrable and essentially bounded functions on $O$, respectively. Without risk of confusion we write $\|\cdot\|_{L^{\infty}(O)}=\|\cdot\|_{\infty}$. We denote by $\|L\|$ the  operator norm of a bounded linear operator $L$ between Banach spaces. Given two sets of real-valued functions $F$ and $\tilde F$, we define $F\cdot\tilde F:=\{f\tilde f: f\in F,\ \tilde f\in\tilde F\}$, and by $\text{Span}\{F\}$ we mean the set of all linear combinations of functions in $F$. The notation we use for an $E$-valued stochastic process started at $x\in E$ is $X^x=\{X^x(s)\}_{s\ge0}$. Note that the symbol $t$ will often be used to denote the starting point of a stochastic process with state space $E\subset \mathbb R$. By a \emph{ strongly continuous contraction semigroup} $P$ we mean a collection of linear operators $P_s:B\to B$, $s\ge0$, where $B$ is a Banach space, such that $P_{s+r}=P_sP_r$, for every $s,r\ge 0$, $P_0$ is the identity operator, $\lim_{s\downarrow 0}P_sf=f$ in $B$, for every $f\in B$,  and $\sup_s\|P_s\|\le1$. The generator of the semigroup $P$ is defined as the pair $(\mathcal L,\text{Dom}(\mathcal L))$, where $\text{Dom}(\mathcal L):=\{f\in B: \mathcal L f:=\lim_{s\downarrow0}s^{-1}(P_sf-f) \text{ exists in }B\}$. We say that a set $C\subset \text{Dom}(\mathcal L)$ is a \emph{core for} $(\mathcal L,\text{Dom}(\mathcal L))$ if the generator equals the closure of the restriction of $\mathcal L$ to $C$. We say that a set $C\subset B$ is \emph{invariant under} $P$ if $P_sC\subset C$ for every $s>0$. If a set $C$ is invariant under $P$ and a core for  $(\mathcal L,\text{Dom}(\mathcal L))$, then we say that $C$ is an \emph{invariant core for} $(\mathcal L,\text{Dom}(\mathcal L))$.  For a given $\lambda\ge0$ we define the \emph{resolvent of} $P$  by $(\lambda-\mathcal L)^{-1}:=\int_0^\infty e^{-\lambda s}P_s ds$, and recall that for $\lambda >0$, $(\lambda-\mathcal L)^{-1}:B\to \text{Dom}(\mathcal L)$ is a bijection and it solves the abstract resolvent equation 
\begin{equation*}
\mathcal L(\lambda-\mathcal L)^{-1}f=\lambda (\lambda-\mathcal L)^{-1}f-f,\quad f\in B,
\end{equation*}
 see for example \cite[Theorem 1.1]{dynkin1965}. By a \emph{sub-Feller semigroup} we mean a strongly continuous contraction semigroup on any of the Banach spaces of continuous functions defined above such that $P$ preserves non-negative functions. A \emph{Feller semigroup} is a sub-Feller semigroup such that its extension to bounded measurable functions preserves constants. 

\subsection{Fractional derivatives, stable processes and related space-time semigroups}
\begin{definition}\label{def_operators}
For parameters $\beta\in(0,1)$ and $\alpha\in(0,2)$,  we define: the \emph{Marchaud derivative} $D^{\beta}_{\infty}$ by formula (\ref{gen}); the \emph{Caputo derivative} $D^\beta_0$ by
\begin{equation*}
D^\beta_0f(t)=\int_0^{t} (f(t-r)-f(t))\,\frac{\Gamma(-\beta)^{-1}dr}{r^{1+\beta}}+(f(0)-f(t))\int_t^\infty \frac{\Gamma(-\beta)^{-1}dr}{r^{1+\beta}},\quad t>0,
\end{equation*}
and $D^\beta_0f(0)=\lim_{t\downarrow 0}D^\beta_0f(t)$; the \emph{restricted fractional Laplacian} $\flap_\Omega$ by 
\begin{equation*}
\flap_\Omega f(x)=\lim_{\varepsilon\downarrow 0}\int_{ \Omega\backslash B_\varepsilon(x)}(f(y)-f(x))\,\frac{c_{\alpha,d}\,dy}{|x-y|^{d+\alpha}}-f(x)\int_{\mathbb R^d\backslash \Omega}\frac{c_{\alpha,d}\,dy}{|x-y|^{d+\alpha}},\quad x\in\Omega,
\end{equation*}
and $\flap_\Omega f(z)=\lim_{x\to z}\flap_\Omega f(x)$ for $z\in\partial\Omega$, where $c_{\alpha,d}^{-1}=\int_{\mathbb R^d}\frac{1-\cos y_1}{|y|^{d+\alpha}}\,dy$, $|\cdot|$ denotes the Euclidean norm on $\mathbb R^d$ and $B_\varepsilon(x)$ denotes the Euclidean ball of radius $\varepsilon>0$ around $x\in\Omega$.
\end{definition}
We now define several sub-Feller semigroups that relate  to the fractional derivatives in Definition \ref{def_operators} and collect some results relevant for us. For $\beta\in(0,1)$, we denote by $X^\beta=\{X^\beta(s)\}_{s\ge0}$ the standard $\beta$-stable subordinator, and by $p_s^\beta$ the smooth density  of $X^\beta(s)$, $s>0$. 
\begin{definition}
For $\beta\in (0,1)$, we denote by $-X^{t,\beta}=\{-X^{t,\beta}(s):=t-X^\beta(s)\}_{s\ge0}$ the\emph{ inverted $\beta$-stable subordinator started at }$t\in\mathbb R$, characterised by the Laplace transforms $\mathbf E[e^{- X^{0,\beta}(s)k}]=e^{-k^{\beta}s}$, $k,s> 0$. We define the first exit/passage times $\tau_0(t)=\inf\{s>0: t-X^{\beta}(s)\le0\}$, $t\in\mathbb R$. 
\end{definition}

\begin{definition}
For $\alpha\in (0,2)$, $d\in\mathbb N$, we denote by $X^{x,\alpha}=\{X^{x,\alpha}(s)\}_{s\ge0}$ the \emph{rotationally symmetric $\alpha$-stable L\'evy process with values in $\mathbb R^d$, started at }$x\in\mathbb R^d$, with characteristic functions $\mathbf E[e^{ik\cdot X^{0,\alpha}(s)}]=e^{-s|k|^{\alpha}}$, $k\in\mathbb R^d$, $s> 0$.  We define the first exit times $\tau_\Omega(x)=\inf\{s>0:X^{x,\alpha}(s)\notin \Omega\}$, $x\in\mathbb R^d$.

\end{definition}
Recall that the smooth density of $-X^{t,\beta}(s)$, $s>0$, is supported $(-\infty,t)$ and it equals $p_s^\beta(t-\cdot)$, and that the law of $X^{x,\alpha}(s)$ is smooth for each $s>0$ (see for example \cite[page 10]{Bogdan}).

\begin{proposition}\label{prop_csg} Fix $T>0$. For the the inverted $\beta$-stable subordinator $-X^{t,\beta}$, denote the Feller semigroup $P^{\beta,\infty}=\{P^{\beta,\infty}_s\}_{s\ge0}$ on $C_\infty((-\infty,T])$, by $P^{\beta,\infty}_sf(t):=\mathbf E[f(-X^{t,\beta}(s))]$, $s\ge 0$, denote by $(\mathcal L_\beta^\infty,\text{Dom}(\mathcal L_\beta^\infty))$ the generator of $P^{\beta,\infty}$, and recall that $C^1_\infty((-\infty,T])$ is an invariant core for $(\mathcal L_\beta^\infty,\text{Dom}(\mathcal L_\beta^\infty))$ with $\mathcal L_\beta^\infty=-D^\beta_{\infty}$ on $C^1_\infty((-\infty,T])$.
\begin{enumerate}[(i)]
	\item 
Define the absorbed process $-X^{t,\beta}_0$ by
\begin{equation}
-X^{t,\beta}_0(s):=
\left\{ \begin{split}
-X^{t,\beta}(s)&,&\text{if }s<\tau_0(t),\\
0&,&\text{if }s\ge\tau_0(t).
\end{split}
\right.
\label{def_abs}
\end{equation}
Then the process $-X^{t,\beta}_0$ induces a Feller semigroup on $C([0,T])$, denoted by $P^\beta=\{P^{\beta}_s\}_{s\ge 0}$, with generator $(\mathcal L_\beta,\text{Dom}(\mathcal L_\beta))$. Moreover, $C^1([0,T])$ is an invariant core for $(\mathcal L_\beta,\text{Dom}(\mathcal L_\beta))$ and
\begin{equation*}
\mathcal L_\beta= -D^\beta_0\quad\text{on}\quad C^1([0,T]).
\end{equation*}
\item The sub-Feller semigroup $P^{\beta,\text{kill}}:=P^\beta$ on $C_{0}([0,T])$ is the the sub-Feller semigroup induced by the killed version of the process (\ref{def_abs}), and its generator is $(\mathcal L_\beta^{\text{kill}}, \text{Dom}(\mathcal L_\beta^{\text{kill}}))=(\mathcal L_\beta, \text{Dom}(\mathcal L_\beta)\cap\{f(0)=0\})$.  Moreover, $C^{1}_0([0,T])$ is an invariant core for 
$(\mathcal L_\beta^{\text{kill}}, \text{Dom}(\mathcal L_\beta^{\text{kill}}))$ and
\begin{equation*}
\mathcal L_{\beta}^{\text{kill}}= -D^\beta_0\quad\text{on}\quad C^1_0([0,T]).
\end{equation*}
\item The following three identities hold
\begin{equation}
\mathbf E\left[\tau_0(t)\right]=\frac{t^\beta}{\Gamma(\beta+1)},\quad \mathbf E\left[e^{-\lambda \tau_0(t)}\right]=E_\beta(-\lambda t^\beta),\quad t,\lambda\ge0, \text{ and}
\label{ident1}
\end{equation}
\begin{equation}
\int_0^\infty p^\beta_s(t-r)\,ds=\frac{(t-r)^{\beta-1}}{\Gamma(\beta)},\quad t> r.
\label{ident3}
\end{equation}

\item The alternative representation of the Caputo derivative
\begin{equation*}
D^{\beta}_0u(t)=\int_0^tu'(r)\,\frac{(t-r)^{-\beta}dr}{\Gamma(1-\beta)}, \quad \text{for $0<t<T$,}
\end{equation*}
holds if $u\in C([0,T])\cap C^1((0,T))$ and $u'\in L^1((0,T))$.
\end{enumerate}
\end{proposition}

\begin{proof}
\begin{enumerate}[(i)]
	\item It is easy to prove that $P^\beta_s f(t):=\int^t_0 f(r) p^\beta_s(t-r)\,dr+ f(0)\int^0_{-\infty}p^\beta_s(t-r)\,dr$ is a Feller semigroup on $C([0,T])$, and the corresponding process is indeed $-X_0^{t,\beta}$. By using the proof of \cite[Proposition 14]{Vio10}\footnote{We select $c_+ = \Gamma(-\alpha)^{-1}$ and $c_-=0$ in \cite[Proposition 14]{Vio10}. In the statement of \cite[Proposition 14]{Vio10} it is required that $F\in C^2([0,\infty))$, but $F\in C^1([0,\infty))$ is enough.   }, it holds that $C^1([0,T])\subset \text{Dom}(\mathcal L_\beta)$, and that $\mathcal L_\beta=-D_0^\beta$ on $C^1([0, T])$. To prove that $C^1([0, T])$ is invariant under $P^\beta$, we directly compute for $g\in C^1([0, T])$, $t\in (0, T)$ and $s > 0$, 
	\begin{align*}
	\partial_t P_s^\beta g(t)&=\partial_t\left(\int_0^tg(t-r)p^\beta_s(r)\,dr+g(0)\int_{-\infty}^{-t} p^\beta_s(-r)\,dr\right)\\
	&=\int_0^tg'(t-r)p^\beta_s(r)\,dr \pm g(0)p^\beta_s(t).
	\end{align*} 
	Then $C^1([0,T])$ is a dense subspace of $\text{Dom}(\mathcal L_\beta)$ which is invariant under $P^\beta$, and so it is a core for $(\mathcal L_\beta,\text{Dom}(\mathcal L_\beta))$ by \cite[Lemma 1.34]{Schilling}.
	\item Similarly  to part (i), it can be shown that $P^{\beta,\text{kill}}_s f(t)=\int^t_0 f(r) p^\beta_s(t-r)\,dr$. To show  $\text{Dom}(\mathcal L_\beta)\cap\{f(0)=0\}\subset \text{Dom}(\mathcal L_\beta^{\text{kill}})$ , let $f\in \text{Dom}(\mathcal L_\beta)\cap\{f(0)=0\}$, then for some $\lambda>0$, let $g\in C([0,T])$ such that
	\begin{equation*}
	f(t)=\int_0^\infty e^{-\lambda s}P^{\beta}_s g(t)\,ds,\quad\text{and}\quad g(0)\frac{1}{\lambda}=\int_0^\infty e^{-\lambda s}P^{\beta}_s g(0)\,ds=f(0)=0,
	\end{equation*}
	and so $g\in C_0([0,T])$. As $P^{\beta}_s=P^{\beta,\text{kill}}_s$ on $ C_0([0,T])$, it follows that $f\in  \text{Dom}(\mathcal L_\beta^{\text{kill}})$. The  inclusion $\text{Dom}(\mathcal L_\beta)\cap\{f(0)=0\}\supset \text{Dom}(\mathcal L_\beta^{\text{kill}})$ is immediate using $P^{\beta}_s=P^{\beta,\text{kill}}_s$ on $ C_0([0,T])$. By equating a resolvent equation, it follows that $\mathcal L_\beta^{\text{kill}}=\mathcal L_\beta$ on $\text{Dom}(\mathcal L_\beta^{\text{kill}})$. Invariance of $C^1_0([0,T])$ can be proven as in part (i). The last statement now follows from part (i).
	\item The first identity  follows from the third identity (\ref{ident3}). The second identity follows by \cite[Theorem 2.10.2]{zolotarev}. To prove the third identity (\ref{ident3}), recall that
\begin{equation*} \label{E:ps-beta}
p_s^{\beta} (t-r) = s^{-1/\beta} p^{\beta}_1 (s^{-1/\beta} (t-r)), \quad t>r,
\end{equation*}
 and then compute
  \begin{align*}
\int_0^{\infty} p_s^{\beta}(t,r)\,ds &= (t-r)^{\beta-1}  \int_0^{\infty} u^{-1/\beta} p^{\beta}_1 (u^{-1/\beta})\,du =  (t-r)^{\beta-1}\frac{1}{\Gamma(\beta)}, \label{beta-density}
\end{align*}
 using  the Mellin transform of the $\beta$-stable density $p^{\beta}_1$ for the last equality (see for example \cite[Theorem 2.6.3]{zolotarev}). 
	\item This is a standard computation and we omit it.
\end{enumerate}
\end{proof}

We say that a bounded open set $\Omega\subset\mathbb R^d$ is a \emph{regular set} if  $\Omega$ satisfies the exterior cone condition at every point $\partial\Omega$, i.e. for each $x\in\partial\Omega$ there exists a finite right circular open cone $V_x$ with vertex $x$, such that $V_x\subset \Omega^c$  (see \cite[end of Section 4]{MeerChen}). From now on   $\Omega$ is always a regular set.
\begin{proposition}\label{prop_fl}
Define the sub-process  $X^{x,\alpha}_\Omega$ started at $x\in\Omega$ by 
\begin{equation*}
X^{x,\alpha}_\Omega(s):=
\left\{
\begin{split}
X^{x,\alpha}(s),&\quad  s<\tau_\Omega(x),\\
\text{cemetery},&\quad   s\ge\tau_\Omega(x),
\end{split}
\right.
\end{equation*} 
\begin{enumerate}[(i)]
	\item Then $X^{x,\alpha}_\Omega$ induces  a sub-Feller semigroup on $C_{\partial\Omega}(\Omega)$, which we denote by $P^\Omega =\{P^{\Omega}_s\}_{s\ge0}$, and we denote its generator by $(\mathcal L_\Omega, \text{Dom}(\mathcal L_\Omega))$. Moreover if $u\in \text{Dom}(\mathcal L_\Omega)$ then there exists a sequence $u_n\in  C_{\partial\Omega}(\Omega)\cap C^2(\Omega)$ such that $u_n\to u$ uniformly and $\flap_\Omega u_n\to \mathcal L_\Omega u$ uniformly on compact subsets of $\Omega$. The transition density of $X^{x,\alpha}_\Omega(s)$, denoted by $p^\Omega_s(x,y),$ is jointly continuous in $x$ and $y$, for every $s>0$. 
	\item For every $u\in \text{Dom}(\mathcal L_\Omega)$ and $\varphi\in C_c^2(\Omega)$ it holds
		\begin{equation}
	\int_\Omega \mathcal L_\Omega u \varphi\, dx=	\int_\Omega u \flap_\Omega  \varphi \,dx.
	\label{dual_dflg}
	\end{equation}
	\item The semigroup $P^\Omega$ induces a strongly continuous  contraction semigroup on $L^2(\Omega)$, and we denote its generator by $(\mathcal L_{\Omega,2},\text{Dom}(\mathcal L_{\Omega,2}))$. Moreover  there exists a sequence of positive numbers $0 < \lambda_1 < \lambda_2 \le \lambda_3 \le \dots, $ and an orthonormal basis $\{\psi_n \}_{n\in\mathbb N}$ of $L^2(\Omega)$, so that $P^\Omega_s\psi_n = e^{-\lambda_n s }\psi_n$ in  $L^2(\Omega)$, for every $n\in\mathbb N,\ s>0$. For $k\ge1$,  we denote by $ \text{Dom}(\mathcal L_{\Omega,2}^k)$ the subset of $ L^2(\Omega)$ such that $\|f\|_{\mathcal L_{\Omega,2}^k}:=\left(\sum_{n=1}^\infty \lambda_n^{2k}\langle f,\psi_n\rangle^2\right)^{1/2}<\infty$. Moreover, $P^\Omega$ on $C_{\partial\Omega}(\Omega)$ has the same set of eigenvalues and eigenfunctions as $P^\Omega$ on $L^2(\Omega)$.
\end{enumerate}
\end{proposition}
\begin{proof}
(i) The first two statements are a consequence of   \cite[Lemma 2.2 and Theorem 2.7]{MeerB}. The last statement follows by the strong Markov property along with joint continuity of the transition densities of $X^{x,\alpha}$ (see for example \cite[Section 4]{MeerChen}).

(ii) The operator $\flap_\Omega$ is self-adjoint in the sense that 
	\begin{equation}
	\int_\Omega \flap_\Omega u \varphi\, dx=	\int_\Omega u \flap_\Omega  \varphi\, dx,
	\label{dual_dfl}
	\end{equation}
		if $\varphi\in C_c^2(\Omega)$ and $u\in C_{\partial\Omega}(\Omega)\cap C^2(\Omega)$. Now use the approximating sequence from part (i) of the current proposition to conclude.

(iii) These results can be found in \cite[Section 4]{MeerChen} and references therein.

\end{proof}

In the next lemma we construct three sub-Feller semigroups by combining in space-time the sub-Feller semigroups defined so far. We combine them in a way that allows us to describe the newly constructed space-time generator as the closure of the sum of the time and space generators. This is how we give meaning to the boundary value problem viewpoint formally presented in (\ref{bvpp}).
\begin{lemma}\label{thm_semigroups} Consider the four tuples
 \begin{align*}
&(P^{\beta,\infty}, C_\infty((-\infty,T]), \mathcal L_\beta^\infty, \text{Dom}(\mathcal L_\beta^\infty)), &(&P^\beta, C([0,T]), \mathcal L_\beta, \text{Dom}(\mathcal L_\beta)),\\
&(P^{\beta,\text{kill}}, C_{0}([0,T]), \mathcal L_\beta^{\text{kill}}, \text{Dom}(\mathcal L_\beta^{\text{kill}})), &(&P^\Omega, C_{\partial\Omega}(\Omega), \mathcal L_\Omega, \text{Dom}(\mathcal L_\Omega)),
\end{align*}
defined in Proposition \ref{prop_csg},  Proposition \ref{prop_csg}-(i), Proposition \ref{prop_csg}-(ii) and Proposition \ref{prop_fl}-(i), respectively.  Let $\mathcal C_\beta^{\infty}$, $\mathcal C_\beta $, $\mathcal C_\beta^{\text{kill}} $ and $\mathcal C_\Omega $ be invariant cores for $(\mathcal L_\beta^\infty, \text{Dom}(\mathcal L_\beta^\infty))$, $(\mathcal L_\beta, \text{Dom}(\mathcal L_\beta))$, $(\mathcal L_\beta^{\text{kill}}, \text{Dom}(\mathcal L_\beta^{\text{kill}}))$ and $(\mathcal L_\Omega, \text{Dom}(\mathcal L_\Omega))$, respectively. 
\begin{enumerate}[(i)]
	\item 
Then $P^{\beta,\Omega}=\{P^\beta_sP^\Omega_s\}_{s\ge 0}$ is a sub-Feller semigroup on $C_{\partial\Omega}([0,T]\times\Omega)$. The generator $(\mathcal L_{\beta,\Omega},\text{Dom}(\mathcal L_{\beta,\Omega}))$ of $P^{\beta,\Omega}$ is the closure of
$$
(\mathcal L_\beta+\mathcal L_\Omega, \text{Span}\left\{\mathcal C_\beta\cdot \mathcal C_\Omega\right\} )\quad\text{in}\quad C_{\partial\Omega}([0,T]\times\Omega),
$$
 where $P^\beta$ and $\mathcal L_\beta$ act on the $[0,T]$-variable, and $P^\Omega$ and $\mathcal L_\Omega$ act on the $\Omega$-variable.

 \item 
Then $P^{\beta,\Omega,\text{kill}}=\{P^{\beta,\text{kill}}_sP^\Omega_s\}_{s\ge 0}$ is a sub-Feller semigroup on $C_{0,\partial\Omega}([0,T]\times\Omega)$. The generator $(\mathcal L_{\beta,\Omega}^{\text{kill}},\text{Dom}(\mathcal L_{\beta,\Omega}^{\text{kill}}))$ of $P^{\beta,\Omega,\text{kill}}$ is the closure of
$$
(\mathcal L_\beta^{\text{kill}}+\mathcal L_\Omega, \text{Span}\{\mathcal C_\beta^{\text{kill}}\cdot \mathcal C_\Omega\} )\quad\text{in}\quad C_{0,\partial\Omega}([0,T]\times\Omega),
$$
 where $P^{\beta,\text{kill}}$ and $\mathcal L_\beta^{\text{kill}}$ act on the $[0,T]$-variable, and $P^\Omega$ and $\mathcal L_\Omega$ act on the $\Omega$-variable. 

 \item 
Then $P^{\beta,\Omega,\infty}=\{P^{\beta,\infty}_sP^\Omega_s\}_{s\ge 0}$ is a sub-Feller semigroup on $C_{\infty,\partial\Omega}((-\infty,T]\times\Omega)$. The generator $(\mathcal L_{\beta,\Omega}^{\infty},\text{Dom}(\mathcal L_{\beta,\Omega}^{\infty}))$ of $P^{\beta,\Omega,\infty}$ is the closure of
$$
(\mathcal L_\beta^{\infty}+\mathcal L_\Omega, \text{Span}\{\mathcal C_\beta^{\infty}\cdot \mathcal C_\Omega\} )\quad\text{in}\quad C_{\infty,\partial\Omega}((-\infty,T]\times\Omega),
$$
 where $P^{\beta,\infty}$ and $\mathcal L_\beta^{\infty}$ act on the $(-\infty,T]$-variable, and $P^\Omega$ and $\mathcal L_\Omega$ act on the $\Omega$-variable. 

\item It holds that $P_s^{\beta,\Omega}=P_s^{\beta,\Omega,\text{kill}}$ $\text{on }C_{0,\partial\Omega}([0,T]\times \Omega),$ $ 
\mathcal L_{\beta,\Omega}=\mathcal L_{\beta,\Omega}^{\text{kill}}$ $\text{on }\text{Dom}(\mathcal L_{\beta,\Omega}^{\text{kill}}),$ and $
\text{Dom}(\mathcal L_{\beta,\Omega}^{\text{kill}})=\text{Dom}(\mathcal L_{\beta,\Omega})\cap \{f(0)=0\}$.
\end{enumerate}
\end{lemma}
\begin{proof}
The proofs of (i), (ii) and (iii) can be found in Appendix \ref{proof-i-ii-iii}.

(iv) The first claim is an immediate consequence of $P^{\beta,\text{kill}}=P^\beta$ on $C_{0}([0,T])$.
The second claim follows from the third by considering a resolvent equation. To prove the third claim, we show the equivalent statement
\[
\text{Dom}(\mathcal L_{\beta,\Omega}^{\text{kill}})\subset \text{Dom}(\mathcal L_{\beta,\Omega}),\quad\text{and}\quad\text{if }u\in \text{Dom}(\mathcal L_{\beta,\Omega}),\text{ then }u-u(0)\in \text{Dom}(\mathcal L_{\beta,\Omega}^{\text{kill}}).
\]
The first inclusion is immediate using $P_s^{\beta,\Omega}=P_s^{\beta,\Omega,\text{kill}}, \text{ on }C_{0,\partial\Omega}([0,T]\times \Omega)$. For the second part, let $u\in \text{Dom}(\mathcal L_{\beta,\Omega})$ and consider its resolvent representation for some $\lambda>0$ and $g\in C_{\partial\Omega}([0,T]\times \Omega)$. Then
\[
u(0,x)=\int_0^\infty e^{-\lambda s}P^\beta_sP^\Omega_sg(0,x)\,ds=\int_0^\infty e^{-\lambda s}P^\beta_sP^\Omega_s(g(0))(t,x)\,ds,
\] 
as $P^\beta_sg(0,x)=P^\beta_s(g(0))(t,x)$. Now consider
\begin{align*}
u(t,x)-u(0,x)&=\int_0^\infty e^{-\lambda s}P^\Omega_sP^\beta_s(g-g(0))(t,x)\,ds\\
&=\int_0^\infty e^{-\lambda s}P^\Omega_sP^{\beta,\text{kill}}_s(g-g(0))(t,x)\,ds\in \text{Dom}(\mathcal L_{\beta,\Omega}^{\text{kill}}),
\end{align*}
where we use the fact that  $P^{\beta,\text{kill}}=P^{\beta}$ on $C_{0,\partial\Omega}([0,T]\times \Omega)$ and that $g-g(0)\in C_{0,\partial\Omega}([0,T]\times \Omega)$.
\end{proof}

\begin{remark}
Note that 
\[
(-\mathcal L_{\beta,\Omega}^{\text{kill}})^{-1}g(t,x)=\int_0^\infty P_s^{\beta,\Omega}g(t,x)\, ds=\mathbf E\left[\int_0^{\tau_0(t) \wedge \tau_\Omega(x)}g\left(-X^{t,\beta}(s),X^{x,\alpha}(s)\right)ds\right],
\]
for $g\in C_{0,\partial\Omega}([0,T]\times \Omega)$. Also, from now on we might write $\tau_{t,x}$ for $\tau_0(t) \wedge \tau_\Omega(x)$.
\end{remark}

\section{Stochastic weak solution for problem (\ref{postRL})}\label{sec_ws}

\subsection{Definition of weak solution}
Define the operator 
\begin{equation*}
-D^{\beta,*}_0\varphi (s):=  \partial_sI^{1-\beta}_{T}\varphi(s) +\delta_0(ds) I^{1-\beta}_{T}\varphi(0),
\end{equation*}
where $\delta_0$ is the delta-measure at $0$, and the Riemann-Liouville integral $I^{1-\beta}_{T}$ is defined as
\[
I^{1-\beta}_{T}f(s):=\int_s^Tf(t)\,\frac{(t-s)^{-\beta}dt}{\Gamma(1-\beta)},\quad s<T.
\]
 In the current section only the pairing $\langle\cdot ,\cdot\rangle$ is defined as
\begin{equation*}
\langle f, g\rangle := \int_0^T\int_\Omega f(t,x)g(t,x) \,dx\,dt.
\end{equation*}
\begin{definition}\label{def_weaksolution}
Let $f\in L^\infty((0,T)\times\Omega)$ and $\phi_0\in C_{\partial\Omega}(\Omega)$. A function $u\in L^{2}((0,T)\times \Omega)$ is said to be a \emph{weak solution to problem }(\ref{postRL}) if
\begin{equation}
\langle u, (-D^{\beta,*}_0+\flap_\Omega)\varphi\rangle =\langle -f, \varphi\rangle,\quad\text{for every }\varphi\in C^{1,2}_c((0,T)\times \Omega),
\label{wsdual}
\end{equation}
and $u(t)\to \phi_0$ a.e. as $t\downarrow 0$.

\end{definition}

The next proposition motivates Definition \ref{def_weaksolution}.
\begin{proposition}\label{prop_Caputodual}
Let $\varphi\in C_c^1((0,T))$ and $u\in C([0,T]) \cap C^1((0,T))$ such that $u'\in L^1((0,T))$. Then
\begin{equation*}
\int_0^TD^{\beta}_0u(t)\varphi(t)\,dt =- \int_0^Tu(t)\left(\partial_t I^{1-\beta}_{T}\varphi(t)\right)\,dt -u(0) I^{1-\beta}_{T}\varphi(0).
\end{equation*}
\end{proposition}
\begin{proof}
Using Proposition \ref{prop_csg}-(iv), Fubini's Theorem and  integration by parts, compute
\begin{align*}
\int_0^TD^\beta_0u(t)\varphi(t)\,dt&=\int_{\mathbb R} \int_{\mathbb R} u'(s)\frac{(t-s)^{-\beta}}{\Gamma(1-\beta)}\varphi(t)\mathbf1_{\{0< t< T\}}\mathbf1_{\{0< s< t\}}\,ds\,dt\\
&=\int_{\mathbb R} u'(s) \mathbf1_{\{0< s<T\}}\left(\int_s^T\frac{(t-s)^{-\beta}}{\Gamma(1-\beta)}\varphi(t)\,dt\right)\,ds\\
&=\int_0^T u'(s) I^{1-\beta}_{T}\varphi(s)\,ds\\
&=-\int_0^Tu(s) \partial_sI^{1-\beta}_{T}\varphi(s)\,ds- u(0) I^{1-\beta}_{T}\varphi(0).
\end{align*}
\end{proof}

From Proposition \ref{prop_Caputodual} and the identity in (\ref{dual_dfl}), it is straightforward to prove the following lemma.
\begin{lemma}\label{lem_dualonsmooth}
Let  $\varphi\in C_c^{1,2}((0,T)\times \Omega)$ and $u\in C_{\partial\Omega}([0,T]\times \Omega) \cap C^{1,2}((0,T)\times \Omega)$ such that $\partial_tu\in L^1((0,T)\times\Omega)$. Then 
\begin{equation*}
\langle u,  (-D^{\beta,*}_0+\flap_\Omega)\varphi\rangle = \langle (-D^{\beta}_0+\flap_\Omega) u,  \varphi\rangle.
\end{equation*}
\end{lemma}

\subsection{Existence of a weak solution}

Following  \cite{HKT17}, we define two auxiliary notions of solution for problem (\ref{postRL}), starting from the abstract evolution equation
\begin{equation}
\mathcal L_{\beta,\Omega} u = -f \text{ on }(0,T]\times\overbar\Omega,\quad u=\phi_0 \text{ on }\{0\}\times\overbar\Omega,\quad u\in \text{Dom}(\mathcal L_{\beta,\Omega}).
\label{abstractee}
\end{equation}

\begin{definition}\label{def:sdog}
Let $f\in C_{\partial\Omega}([0,T]\times \Omega)$ and $\phi_0\in \text{Dom}(\mathcal L_\Omega)$ such that $f(0)=-\mathcal L_\Omega \phi_0$. We say that a function $u\in C_{\partial\Omega}([0,T]\times \Omega)$ is a \emph{solution in the domain of the generator to problem} (\ref{postRL}) if  $u$ satisfies  (\ref{abstractee}).
\end{definition}
The next solution concept for problem (\ref{postRL}) is defined as a pointwise approximation of solutions  in the domain of the generator $\{u_n\}_{n\in\mathbb N}$ such that the approximating forcing term $\{f_n\}_{n\in\mathbb N}$ satisfies a dominated convergence type of condition.
\begin{definition}\label{def_gs}
Let $f\in B([0,T]\times \overbar\Omega)$ and $\phi_0\in \text{Dom}(\mathcal L_\Omega)$. We say that a function $u\in B([0,T]\times \overbar\Omega)$ is a \emph{generalised solution  to problem} (\ref{postRL}) if 
$$
u = \lim_{n\to\infty}u_n \quad\text{pointwise},
$$ where each $u_n$ is the  solution in the domain of the generator for a corresponding forcing term $f_n\in C_{\partial\Omega}([0,T]\times \Omega)$ such that 
$$
f_n\to f\text{ a.e. on }(0,T]\times \Omega,\quad\sup_n \|f_n\|_\infty<\infty,\quad\text{and}\quad f_n(0)=-\mathcal L_\Omega \phi_0\text{ for each }n\in\mathbb N.
$$
\end{definition}
\begin{remark} Any generalised solution must satisfy the boundary conditions $u=0$ on $[0,T]\times\partial\Omega$ and $u=\phi_0$ on  $\{0\}\times\Omega$. 
\end{remark}

\begin{lemma}\label{lem_sdoggs} Let $\phi_0\in \text{Dom}(\mathcal L_\Omega)$. Then
\begin{enumerate}[(i)]
	\item If $f+\mathcal L_\Omega \phi_0\in C_{0,\partial\Omega}([0,T]\times \Omega)$, then there exists a unique  solution in the domain of the generator to problem (\ref{postRL}).
	\item If $f\in B([0,T]\times \overbar\Omega)$, then there exists a unique generalised  solution  to problem (\ref{postRL}).
	\item Both solutions in part (i) and (ii) allow the stochastic representation (\ref{SRpostRL}).
	\end{enumerate}
\end{lemma}
\begin{proof}
(i) Observe that the potential $(-\mathcal L_{\beta,\Omega}^{\text{kill}})^{-1}$ maps $C_{0,\partial\Omega}([0,T]\times \Omega)$ to itself. This follows from $P^{\beta,\Omega,\text{kill}}_sg\in C_{0,\partial\Omega}([0,T]\times \Omega)$ for $g\in C_{0,\partial\Omega}([0,T]\times \Omega)$, $s\ge0$, and Dominated Convergence Theorem (DCT) with dominating function $G(s):=\|g\|_\infty\mathbf P[s<\tau_0(T)]$. Note that we use the first identity in (\ref{ident1}) to prove that $G\in L^1((0,\infty))$. The potential $(-\mathcal L_{\beta,\Omega}^{\text{kill}})^{-1}$  is also bounded by the inequality
\begin{equation*}
\left|(-\mathcal L_{\beta,\Omega}^{\text{kill}})^{-1}g(t,x)\right|\le \|g\|_\infty\mathbf E\left[\tau_0(T)\right],\quad g\in C_{0,\partial\Omega}([0,T]\times \Omega).
\end{equation*}
 It then follows by \cite[Theorem 1.1']{dynkin1965} that $\bar u:=(-\mathcal L_{\beta,\Omega}^{\text{kill}})^{-1}(f+\mathcal L_\Omega \phi_0)$ is the unique solution to the abstract evolution equation 
\begin{equation}
\mathcal L_{\beta,\Omega}^{\text{kill}} \bar u = -(f+\mathcal L_\Omega \phi_0) \text{ on }(0,T]\times\overbar\Omega,\quad \bar u=0 \text{ on }\{0\}\times\overbar\Omega,\quad\text{and  $\bar u\in \text{Dom}(\mathcal L_{\beta,\Omega}^{\text{kill}})$}.
\label{aaa}
\end{equation}
It is now enough to show that $\bar u$ satisfies (\ref{aaa}) if and only if $u=\bar u+\phi_0$ satisfies (\ref{abstractee}). For the `if' direction, let $u\in \text{Dom}(\mathcal L_{\beta,\Omega})$  satisfy (\ref{abstractee}). Note that $u(0)=\phi_0$. Then $\bar u:= u-\phi_0\in \text{Dom}(\mathcal L_{\beta,\Omega}^{\text{kill}})$, and $\mathcal L_{\beta,\Omega} \bar u=\mathcal L_{\beta,\Omega}^{\text{kill}} \bar u$, by Lemma \ref{thm_semigroups}-(iv). So we can compute 
\begin{align*}
 \mathcal L_{\beta,\Omega}^{\text{kill}}\bar u = \mathcal L_{\beta,\Omega}( u-\phi_0)= \mathcal L_{\beta,\Omega}u - \mathcal L_\Omega \phi_0=-f - \mathcal L_\Omega \phi_0,
\end{align*}
where we use 
$$
\mathcal L_{\beta,\Omega}1\phi_0=(\mathcal L_\beta+\mathcal L_\Omega)1\phi_0=\mathcal L_\Omega \phi_0,
$$
 from Lemma \ref{thm_semigroups}-(i) taking the invariant cores $\mathcal C_\beta =\text{Dom}(\mathcal L_\beta)$ and $\mathcal C_\Omega=\text{Dom}(\mathcal L_\Omega)$ (recalling that $\mathcal L_\beta1=0$). For the `only if' direction, let $\bar u$ satisfy (\ref{aaa}), and define $u:=\bar u +\phi_0$. Then with the same justifications as just above, compute
\begin{align*}
 \mathcal L_{\beta,\Omega} u =\mathcal  L_{\beta,\Omega}^{\text{kill}}\bar u+\mathcal L_{\beta,\Omega}\phi_0=-(f+\mathcal L_\Omega \phi_0)+\mathcal L_{\Omega}\phi_0=-f.
\end{align*}
It follows that $$
u = (-\mathcal L_{\beta,\Omega}^{\text{kill}})^{-1}(f+\mathcal L_\Omega \phi_0)+ \phi_0.
$$
(ii) Let $f\in B([0,T]\times\overbar\Omega)$. Then $f+\mathcal L_\Omega \phi_0\in B([0,T]\times\overbar\Omega)$. Now take a sequence $\{\tilde f_n\}_{n\mathbb N}\in C_{0,\partial\Omega}([0,T]\times \Omega) $ such that $\tilde f_n\to f+\mathcal L_\Omega \phi_0$ a.e., and $\sup_n\|\tilde f_n\|_\infty<\infty$. Define $f_n:=\tilde f_n -\mathcal L_\Omega \phi_0$ for each $n\in\mathbb N$ and note that $f_n\to f$ a.e., $\sup_n\|\tilde f_n\|_\infty<\infty$ and $f_n(0)=-\mathcal L_\Omega \phi_0$, as required by Definition \ref{def_gs}. Now, for each $f_n$ consider the stochastic representation of the respective solution in the domain  of the generator
\[
u_n(t,x) =\mathbf E\left[\int_0^{\tau_{t,x}} f_n\left(-X^{t,\beta}(s),X^{x,\alpha}(s)\right) ds\right]+\mathbf E\left[\int_0^{\tau_{t,x}}\mathcal L_\Omega \phi_0\left(X^{x,\alpha}(s)\right) ds\right]+\phi_0(x).
\]
 Fix $(t,x)\in (0,T]\times\Omega$. Using absolute continuity with respect of Lebesgue measure of the laws of $-X^{t,\beta}(s)$ and $X^{x,\alpha}_\Omega(s)$ for each $s>0$, and the bound $\mathbf E\left[{\tau_{t,x}}\right]\le \mathbf E\left[{\tau_0(t)}\right]<\infty$, we can apply DCT twice  to obtain as $n\to\infty$
\begin{align*}
\mathbf E\left[\int_0^{\tau_{t,x}} f_n\left(-X^{t,\beta}(s),X^{x,\alpha}(s)\right) ds\right]&= \int_0^\infty P_s^{\beta,\text{kill}}P_s^\Omega f_n(t,x)\,ds\\
&\to \int_0^\infty P_s^{\beta,\text{kill}}P_s^\Omega f(t,x)\,ds\\
&= \mathbf E\left[\int_0^{\tau_{t,x}}f\left(-X^{t,\beta}(s),X^{x,\alpha}(s)\right) ds\right],
\end{align*}
using as a dominating function $G:=\sup_n\|f_n\|_\infty$ to show that for each $s>0$ 
\[
F_n(s):=P_s^{\beta,\text{kill}}P_s^\Omega f_n(t,x)\to P_s^{\beta,\text{kill}}P_s^\Omega f(t,x)=:F(s),
\]
 and the dominating function $G(s):=\sup_n\|f_n\|_\infty\mathbf P[s<\tau_{t,x}]$ to show that 
\[
\int_0^\infty F_n(s)\,ds\to \int_0^\infty F(s)\,ds.
\]
The convergence on $[0,T]\times\partial\Omega\cup\{0\}\times\overbar\Omega$ is trivial.  It follows that a generalised solution $u$ exists and it is  given by
 $$
u = (-\mathcal L_{\beta,\Omega}^{\text{kill}})^{-1}(f+\mathcal L_\Omega \phi_0)+ \phi_0.
$$
Finally, independence of the approximating sequence proves uniqueness.\\

(iii) This is a standard application of Dynkin formula (\cite[Theorem 5.1]{dynkin1965}) using the finite stopping times $\tau_{t,x}$,  $(t,x)\in(0,T]\times\Omega$, namely
\begin{align*}
(-\mathcal L_{\beta,\Omega}^{\text{kill}})^{-1}(\mathcal L_\Omega \phi_0)(t,x) = \mathbf E\left[\int_0^{\tau_{t,x}}\mathcal L_{\beta,\Omega} \phi_0\left(X^{x,\alpha}(s)\right) ds\right]
=\mathbf E\left[\phi_0(X^{x,\alpha}(\tau_{t,x})) \right]-\phi_0(x).
\end{align*}
\end{proof}

We now show that the dual  of $\mathcal L_{\beta,\Omega}$ is $(-D^{\beta,*}_0+\flap_\Omega)$.
\begin{lemma}\label{lem_sdogandws}
Let $u\in \text{Dom}(\mathcal L_{\beta,\Omega})$. Then 
\begin{equation*}
\langle \mathcal L_{\beta,\Omega}u ,\varphi \rangle =\langle u,(-D^{\beta,*}_0+\flap_\Omega)\varphi\rangle,\quad \text{for every }\varphi\in C_c^{1,2}((0,T)\times\Omega).
\end{equation*}
\end{lemma}
\begin{proof}
By Lemma \ref{thm_semigroups}-(i) and Proposition \ref{prop_csg}-(i) we can pick a sequence 
$$
\{u_n\}_{n\in\mathbb N}\subset \text{Span}\left\{C^1([0,T])\cdot \text{Dom}(\mathcal L_\Omega)\right\},
$$ such that $u_n\to u$ and $\mathcal L_{\beta,\Omega}u_n\to \mathcal L_{\beta,\Omega}u$ in $C_{\partial\Omega}([0,T]\times \Omega)$, with the additional property 
\begin{equation}
\mathcal L_{\beta,\Omega}u_n =(-D^\beta_0+\mathcal L_\Omega)u_n,\quad\text{for every }n\in\mathbb N.
\label{ident}
\end{equation}
 Hence, for every $\varphi\in C_c^{1,2}((0,T)\times\Omega)$, as $n\to\infty$
\begin{equation*}
\langle \mathcal L_{\beta,\Omega}u ,\varphi \rangle \leftarrow \langle \mathcal L_{\beta,\Omega}u_n ,\varphi \rangle =\langle u_n,(-D^{\beta,*}_0+\flap_\Omega)\varphi\rangle\to \langle u,(-D^{\beta,*}_0+\flap_\Omega)\varphi\rangle,
\end{equation*}
where we use DCT for both limits, and for the equality we use the identity (\ref{ident}) along with Proposition \ref{prop_Caputodual} and the  dual identity in Proposition \ref{prop_fl}-(ii).
\end{proof}

We now combine Lemma \ref{lem_sdogandws} with the notion of generalised solution to obtain the main theorem of this section.
\begin{theorem}\label{thm_wspostRL}
Let $f\in L^\infty((0,T)\times \Omega)$ and $\phi_0\in C_{\partial\Omega}(\Omega)$. Then the function $u\in B([0,T]\times \overbar\Omega)$ defined in (\ref{SRpostRL})  is a weak solution to problem (\ref{postRL}).
\end{theorem}
\begin{proof}
Assume for the moment that $\phi_0\in \text{Dom}(\mathcal L_\Omega)$. By the definition of a generalised solution we can take an approximating sequence of forcing terms $\{f_n\}_{n\in\mathbb N}\subset C_{\partial\Omega}([0,T]\times\Omega)$ such that $f_n\to f$ a.e., $\sup_n\|f_n\|_\infty<\infty$, and the respective solutions in the domain of the generator $\{u_n\}_{n\in\mathbb N}$ satisfy
\[
u_n(0)=\phi_0\text{ for all }n\in\mathbb N,\quad u_n\to u\text{ pointwise on }[0,T]\times \Omega,\quad \sup_n\|u_n\|_\infty<\infty,
\]  
where the last property is an immediate consequence of the stochastic  representation (\ref{SRpostRL}).
Hence, we obtain for every $\varphi \in C^{1,2}_c((0,T)\times\Omega)$, as $n\to\infty$
\[
\langle -f,\varphi \rangle\leftarrow\langle -f_n,\varphi \rangle = \langle \mathcal L_{\beta,\Omega}u_n,\varphi \rangle  = \langle u_n,(-D^{\beta,*}_0+\flap_\Omega)\varphi \rangle \to \langle u,(-D^{\beta,*}_0+\flap_\Omega)\varphi\rangle,
\]
where we applied DCT for both limits, the first equality is due to the $u_n$'s being solutions in the domain of the generator, and the second equality holds as a consequence of Lemma \ref{lem_sdogandws}. \\
Now, for $\phi_{0}\in C_{\partial\Omega}(\Omega)$, let  $\{\phi_{0,n}\}_{n\in\mathbb N}\subset\text{Dom}(\mathcal L_\Omega)$ such that $\phi_{0,n}\to \phi_{0}$ in $C_{\partial\Omega}(\Omega)$. Let $u_n$ be the generalised solution to problem (\ref{preRL}) for $f\in B([0,T]\times \overbar\Omega)$ and $\phi_n\in \text{Dom}( \mathcal L_\Omega)$, and $u$ defined as in (\ref{SRpostRL}). Then $u_n\to u$ pointwise and $\sup_n\|u_n\|_\infty<\infty$, which in turn implies by DCT
\[
\langle -f,\varphi \rangle=\lim_{n\to\infty} \langle u_n,(-D^{\beta,*}_0+\flap_\Omega)\varphi \rangle =\langle u,(-D^{\beta,*}_0+\flap_\Omega)\varphi \rangle.
\] 
It is clear that the result holds for $f\in L^\infty((0,T)\times \Omega)$. Finally, the required convergence of $u$ to the initial condition $\phi_0$ follows by the argument in Remark \ref{rmk_cont}, using the stochastic representation (\ref{SRpostRL}).
\end{proof}

\section{Stochastic classical solution for problem (\ref{postRL})}\label{sec_existence}
\begin{definition}Let $f\in C((0,T]\times\Omega)$ and $\phi_0\in C(\Omega)$. A function $u\in C_{\partial\Omega}([0,T]\times\Omega)\cap C^{1,2}((0,T)\times\Omega)$, such that $|\partial_tu(t,x)|\le Ct^{-\gamma},$ for every $(t,x)\in(0,T]\times\Omega, \text{ for some }\gamma\in(0,1),\ C>0$,  is said to be a \emph{classical solution to problem} (\ref{postRL}) if $u$ satisfies the identities  in (\ref{postRL}),  and for every $x\in\Omega$
$$
\lim_{t\downarrow 0}|u(t,x)-\phi_0(x)|=0.
$$
\end{definition}
In this section  the pairing $\langle \cdot,\cdot\rangle$ is defined as
\[
\langle f,g\rangle:=\int_\Omega f(x)g(x)\,dx.
\]
 The proof of the main theorem of this section (Theorem \ref{thm_sspostRL}),  extends the eigenfunction expansion argument in \cite[Thoerem 5.1]{MeerChen}, using the next  lemma as the key extra ingredient. Define for $\lambda\in\mathbb R\backslash\{0\}$ and $f\in C([0,T])$
\begin{equation*}
\myF{f}{\lambda}(t):=(-\lambda)^{-1}\int_0^t f(r)\partial_t E_\beta(-\lambda (t-r)^\beta)\,dr,\quad t>0.
\end{equation*}

\begin{lemma}\label{lem_identity}
Let $\lambda>0$ and $f\in C([0,T])$. Then
\begin{enumerate}[(i)]
	\item 
\begin{equation*}
\mathbf E\left[\int_0^{\tau_0(t)}e^{-\lambda s} f(-X^{t,\beta}(s))\,ds \right]=\myF{f}{\lambda}(t),\quad t>0.
\end{equation*}

\item The bound  
\begin{equation}
 \left|\myF{f}{\lambda}(t)\right|\le \frac{c}{\lambda}\|f\|_\infty, \quad t>0,
\label{first}
\end{equation}
holds, and if $f\in C^{1}([0,T])$ then
\begin{equation}
\left|\partial_t \myF{f}{\lambda}(t)\right|\le\frac{c}{\lambda}\left( \|f'\|_\infty +f(0)\frac{ \lambda t^{\beta-1}}{1+\lambda t^\beta}\right),\quad t>0,
\label{second}
\end{equation}
for some positive constant $c$.
\end{enumerate}
\end{lemma}

\begin{proof} (i) Given the second identity in (\ref{ident1}), it is enough to prove the equivalent identity 
\begin{equation}
\mathbf E\left[\int_0^{\tau_0(t)}e^{-\lambda s} f(-X^{t,\beta}(s))\,ds \right]+u_0\mathbf E\left[ e^{-\lambda \tau_0(t)} \right]=\myF{f}{\lambda}(t) + u_0E_\beta(-\lambda t^\beta),
\label{whatweprove}
\end{equation}
where $u_0$ is some constant. 
We show that the lhs of (\ref{whatweprove}) is the unique continuous solution to the Caputo initial value problem solved by the rhs of (\ref{whatweprove}).
Let $w\in C_0([0,T])$  such that $w'\in C([0,T])$. Then $u(t):=(\lambda-\mathcal L_\beta)^{-1}w(t) = \mathbf E[\int_0^{\tau_0(t)} e^{-\lambda s}w(-X^{t,\beta}(s))\,ds]$ solves the resolvent equation  
$$
\mathcal L_\beta u =\lambda u - w,\quad u(0)=0,
$$
and $u\in \text{Dom}(\mathcal L_\beta)$, by Proposition \ref{prop_csg}-(i). By the following computation 
\begin{align*}
\partial_t u(t)&=\partial_t \int_0^t w(t-y)\left(\int_0^\infty e^{-\lambda s} p^\beta_s(y)\,ds\right)dy\\
&= w(0)\int_0^\infty e^{-\lambda s} p^\beta_s(t)\,ds+\int_0^t w'(t-y)\int_0^\infty e^{-\lambda s} p^\beta_s(y)\,ds\,dy,&t>0,
\end{align*}
it follows that 
 $u\in C^1_0([0,T])$, and so $\mathcal L_\beta u =-D^{\beta}_0 u$ by Proposition \ref{prop_csg}-(i). Let $u_0\in\mathbb R$. Then $\bar u:= u+u_0 $ is a continuous solution to the Caputo initial value problem 
\begin{align*}
-D^{\beta}_0\bar u=\mathcal L_\beta u-D^{\beta}_0u_0 =\lambda u-w=\lambda \bar u-(w+\lambda u_0),
\end{align*}
with initial value $\bar u(0)=u_0$. By  \cite[Theorem 6.5 and Theorem 7.2]{kai} we obtain $
\bar u =\text{rhs of (\ref{whatweprove}) for $f=w+\lambda u_0$}.$
Now compute
\begin{align*}
\bar u(t)&=\mathbf E\left[\int_0^{\tau_0(t)} e^{-\lambda s}\left(w(-X^{t,\beta}(s))\pm\lambda u_0\right)ds\right] +u_0\\
&=\mathbf E\left[\int_0^{\tau_0(t)} e^{-\lambda s}\left (w(-X^{t,\beta}(s))+\lambda u_0\right)ds\right]-\lambda u_0\frac{\mathbf E\left[ e^{-\lambda \tau_0(t)}\right]-1}{-\lambda} +u_0\\
&=\mathbf E\left[\int_0^{\tau_0(t)} e^{-\lambda s}\left (w(-X^{t,\beta}(s))+\lambda u_0\right)ds\right]+ u_0\mathbf E\left[ e^{-\lambda \tau_0(t)}\right].
\end{align*}
Now, for an arbitrary $f\in C^1([0,T])$, by picking  $w\equiv f-f(0)$ and $u_0 \equiv f(0)\lambda^{-1}$, we obtain the equality (\ref{whatweprove}). A straightforward application of DCT proves the claim for $f\in C([0,T]).$\\

(ii) 
Recall that there exists a constant $c>0$ such that $0\le-\partial_t E_\beta(-\lambda t^\beta)\le c\frac{\lambda t^{\beta-1}}{1+\lambda t^\beta}$ by \cite[Theorem 7.3]{kai} and \cite[Equation (17)]{Kra03}, and  $E_\beta(-\lambda t^\beta)\le \frac{ c}{1+\lambda t^\beta}$. Then
\begin{align*}
\left| (-\lambda)^{-1}\int_0^t f(r) \partial_t E_\beta(-\lambda (t-r)^\beta)\, dr\right|\le&\ \|f\|_{\infty}\frac{1-E_\beta(-\lambda t^\beta)}{\lambda}\le  \|f\|_{\infty} \frac{1+ c}{\lambda}.
\end{align*}
For the second inequality we exploit the smoothness of $f$, computing for $t>0$
\begin{align*}
 \partial_t \myF{f}{\lambda}(t)=&\ (-\lambda)^{-1} \partial_t \left(-\int_0^t f(r)  \partial_rE_\beta(-\lambda  (t-r)^\beta)\, dr\right)\\
 =&\ (-\lambda)^{-1} \partial_t \left(\int_0^t f'(r)  E_\beta(-\lambda  (t-r)^\beta)\, dr-f(t)  +f(0) E_\beta(-\lambda  t^\beta) \right)\\
 =&\ (-\lambda)^{-1}\left(\int_0^t f'(r)  \partial_t E_\beta(-\lambda  (t-r)^\beta)\, dr\pm f'(t)  +f(0) \partial_t E_\beta(-\lambda  t^\beta) \right)\\
 =&\ \myF{f'}{\lambda}(t)-\lambda^{-1}f(0) \partial_t E_\beta(-\lambda  t^\beta).
\end{align*}
Then
\begin{align*}
\left| \partial_t \myF{f}{\lambda}(t)\right|\le &\ \|f'\|_\infty \frac{1+c}{\lambda}+f(0)c\frac{ t^{\beta-1}}{1+\lambda t^\beta}.
\end{align*}
\end{proof}

From the proof of \cite[Theorem 5.1]{MeerChen}, we infer the following lemma.
\begin{lemma}\label{lem_chen_bounds} Working with the notation of  Proposition \ref{prop_fl}-(iii): 
\begin{enumerate}[(i)]
\item  the system of eigenvectors $\{\psi_n\}_{n\in\mathbb N}$ forms an orthonormal basis of $\text{Dom}(\mathcal L_{\Omega,2}^k)\subset L^2(\Omega)$. The corresponding eigenvalues can be ordered so that $\lambda_{n} \leq \lambda_{n+1}$, and also $\lambda_n\le \tilde c_1n^{\alpha/d}$ for some constant $\tilde c_1>0$. Also, for any compact subset $K$ of $\Omega$, $j=0,1,2$, there are  constants $c_1=c_1(K,j,d,\alpha)$ such that
\begin{align}
    \label{eq_eigenvector_bound} |\nabla^j \psi_n(x) | &\leq c_1 \lambda_n ^{(d+2j)/(2\alpha)},
\end{align}
where $c_1(K,0,d,\alpha)$ is independent of $K$.
\item Suppose $\phi_0 \in \operatorname{Dom} (\mathcal L_{\Omega,2}^k)$ for $k > -1 + (3d+4)/(2\alpha)$. Then 
$ N := \sum_{n=1}^\infty \lambda_n^{2k} \langle \phi_0, \psi_n\rangle ^2 < \infty $, and the series 
 \[
\sum_{n=1}^\infty E_\beta (-\lambda_n t^\beta)\langle \phi_0, \psi_n \rangle \psi_n(x)=\mathbf E\left[\phi_0(X^{x,\alpha}(\tau_0(t)))\mathbf 1_{\{\tau_0(t)<\tau_\Omega(x)\}}\right],
 \]
 defines a function in $C_{\partial\Omega}([0,T]\times \Omega)\cap C^{1,2}( (0,T)\times \Omega)$ , with bounds for $j=1,2$,
\begin{alignat}{3}
    & \nonumber\sum_{n=1}^\infty \abs{E_\beta (-\lambda_n t^\beta)\langle \phi_0, \psi_n \rangle \nabla^j \psi_n(x)}
    &&\le( c_2  \sqrt{N})t^{-\beta} \sum_{n=1}^\infty \lambda_n^{(d+4)/(2\alpha) -1-k} < \infty, 
    &&\quad  t>0,\\
    &\nonumber \sum_{n=1}^\infty \abs{\partial_t E_\beta (-\lambda_n t^\beta)\langle \phi_0, \psi_n \rangle \psi_n(x)} 
    &&\leq  c_3 t^{\gamma \beta  - 1}, 
    && \quad x\in \Omega,
\end{alignat}
where $c_2 = c_2(K,j,d,\alpha), c_3 = c_3(\Omega,\alpha)$, and $0\le \gamma \le  1\wedge (4/(2\alpha)-1) $.
\end{enumerate}
\end{lemma}
We will assume that the forcing term $f$ in (\ref{postRL}) belongs to the space of functions
\begin{equation}\label{spaceH}
C^1([0,T];\text{Dom}(\mathcal L_{\Omega,2}^k)):=\left\{f\in C_{\partial\Omega}^1([0,T]\times\overbar \Omega): \sup_t\| f(t)\|_{\mathcal L_{\Omega,2}^k}+ \sup_t\left\| \partial_t f(t)\right\|_{\mathcal L_{\Omega,2}^k}<\infty\right\}.
\end{equation}
 Note that if $f\in C^1([0,T];\text{Dom}(\mathcal L_{\Omega,2}^k))$, then there exists $M>0$ such that for every $n\in\mathbb N$
\begin{equation}
\sup_{t\in[0,T]}|\langle f(t),\psi_n \rangle|\le M\lambda_n^{-k},\quad \text{and}\quad \sup_{t\in[0,T]}\left| \langle \partial_tf(t),\psi_n \rangle\right|\le  M\lambda_n^{-k}.
\label{bounds_pair}
\end{equation}

\begin{remark}\label{rmk_HTk}
The inclusion $\text{Span}\{C^1([0,T])\cdot\text{Dom}(\mathcal L_{\Omega,2}^k)\}\subset C^1([0,T];\text{Dom}(\mathcal L_{\Omega,2}^k))$ is clear. Moreover, if $k\in \mathbb N$, then the inclusion $ C^{1,2k}_{c}([0,T]\times\Omega)\subset C^1([0,T];\text{Dom}(\mathcal L_{\Omega,2}^k))$ holds\footnote{We define  $C^{1,2k}_{c}([0,T]\times\Omega)= C^{1,2k}((0,T)\times\Omega)\cap\{f,\partial_tf \in C([0,T]\times\Omega), \text{supp}\{ f\}\subset [0,T]\times\Omega \text{ is compact} \} $.}. To see this, let $f\in C^{1,2k}_c([0,T]\times\Omega)$ and compute for each $t\in[0,T]$
\begin{align*}
\sum_{n=1}^\infty \lambda^{2k}_n\langle f(t),\psi_n\rangle^2&=\sum_{n=1}^\infty \langle f(t),\mathcal L_{\Omega,2}^k\psi_n\rangle^2=\sum_{n=1}^\infty \langle (\flap_\Omega)^kf(t),\psi_n\rangle^2=\|(\flap_\Omega)^kf(t)\|_{L^2(\Omega)}^2<\infty,
\end{align*}
\end{remark}
where the second equality holds by the same argument at the end the  proof of Theorem \ref{thm_sspostRL}, using $(\flap_\Omega)^mf(t)\in L^2(\Omega)$ for each $t\in[0,T]$ and $m\le k$. Now observe that by DCT the function $t\mapsto\|(\flap_\Omega)^kf(t)\|_{L^2(\Omega)}$ is continuous on $[0,T]$, because $(\flap_\Omega)^kf\in C([0,T]\times\overbar\Omega)$. Repeat the argument for $\partial_t f$ to conclude.

\begin{lemma}\label{lem_nvs}
If $f(t)\in \text{Dom}(\mathcal L_{\Omega,2}^k)$ for $k>-1+(3d+4)/(2\alpha)$, for every $t\in[0,T]$, and $f\in C_{\partial\Omega}([0,T]\times\Omega)$, then
\begin{equation*}
\mathbf E\left[\int_0^{\tau_{t,x}}f\left(-X^{t,\beta}(s), X^{x,\alpha}(s)\right)ds\right]=\sum_{n=1}^\infty \psi_n(x)\myF{\langle f(\cdot),\psi_n \rangle}{\lambda_n}(t).
\end{equation*}
If in addition $f\in C^1([0,T];\text{Dom}(\mathcal L_{\Omega,2}^k))$, then there exists a constant $C$ such that for $t\in(0,T]$
\begin{align}\label{second2}
 \sum_{n=1}^\infty \left| \psi_n(x)\partial_t \myF{\langle f(\cdot),\psi_n \rangle}{\lambda_n}(t)\right|
&\le Ct^{\beta - 1}.
\end{align}
\end{lemma}
\begin{proof}
We justify the following equalities
\begin{align*}
\mathbf E\left[\int_0^{\tau_{t,x}}f\left(-X^{t,\beta}(s), X^{x,\alpha}(s)\right)ds\right]&=\int_0^\infty P_s^{\beta,\text{kill}}P_s^{\Omega} f(t,x)\,ds\\
&=\int_0^\infty P_s^{\beta,\text{kill}}\left(\sum_{n=1}^\infty \langle f(t),\psi_n \rangle \psi_n(x) e^{-s\lambda_n}\right)\,ds\\
&=\sum_{n=1}^\infty \psi_n(x)\int_0^\infty P_s^{\beta,\text{kill}} \langle f(t),\psi_n \rangle  e^{-s\lambda_n}\,ds\\
&=\sum_{n=1}^\infty \psi_n(x)\mathbf E\left[\int_0^{\tau_0(t)} \langle f(-X^{t,\beta}(s)),\psi_n \rangle  e^{-s\lambda_n}\,ds\right]\\
&=\sum_{n=1}^\infty \psi_n(x)\myF{\langle f(\cdot),\psi_n \rangle}{\lambda_n}(t).
\end{align*}
We can apply Fubini's Theorem in the third equality as  $${\sum_{n=1}^\infty |\langle f(t),\psi_n\rangle|\|\psi_n\|_\infty}\le C\sum_{n=1}^\infty n^{(\alpha/d)\left(d/(2\alpha)-k\right)}<\infty,$$ for some constant $C>0$, each $t\ge 0$ and any $ k > 3d/(2\alpha)$, using the bounds in Lemma \ref{lem_chen_bounds}-(i) and in (\ref{bounds_pair}). We apply Lemma \ref{lem_identity}-(i) in the fifth equality as $r\mapsto\langle f(r),\psi_n \rangle\in C([0,T])$ for each $n\in\mathbb N$. The other equalities are clear.\\
For the last claim we use  the  bounds in (\ref{second}),  (\ref{bounds_pair}) and Lemma \ref{lem_chen_bounds}-(i)  to obtain
\begin{align*}
 \sum_{n=1}^\infty \left| \psi_n(x)\partial_t \myF{\langle f(t),\psi_n \rangle}{\lambda_n}(t)\right|&\le\sum_{n=1}^\infty | \psi_n(x)|\frac{c }{\lambda_n}\left(\sup_{r\in[0,T]}\left|\langle \partial_rf(r),\psi_n \rangle\right|+\frac{ \lambda_n t^{\beta-1}}{1+\lambda_n t^\beta} |\langle f(0),\psi_n \rangle|\right)\\
&\le\sum_{n=1}^\infty | \psi_n(x)| \frac{cM\lambda_n^{-k} }{\lambda_n}\left(1+\frac{ \lambda_n t^{\beta-1}}{1+\lambda_n t^\beta}\right)\\
&\le (c_1cM)\sum_{n=1}^\infty \frac{\lambda_n^{d/( 2\alpha)}\lambda_n^{-k} }{\lambda_n}\left(1+\frac{ \lambda_n t^{\beta-1}}{1+\lambda_n t^\beta}\right)\\
        & \leq  ( c_1c M ) t^{\beta - 1} \sum_{n=1}^\infty \lambda_n^{  d/(2\alpha)-k}\\ 
        & \leq (\tilde c_1c_1c M ) t^{\beta - 1} \sum_{n=1}^\infty n^{( \alpha /d )\left( d/(2\alpha)-k\right)}<\infty,
\end{align*}
 for any $ k > 3d/(2\alpha)$, where the constants $\tilde c_1,c_1,c$ and $ M$ follow the notation of the referenced inequalities, and a constant is omitted in the fourth inequality. 
\end{proof}

\begin{theorem}\label{thm_sspostRL} Let $\Omega\subset \mathbb R^d$ be a regular  set. Assume that  $\phi_0\in \text{Dom}(\mathcal L_{\Omega,2}^k)$, and  $f\in C^1([0,T];\text{Dom}(\mathcal L_{\Omega,2}^k))$ for some $k>-1+(3d+4)/(2\alpha)$, where $C^1([0,T];\text{Dom}(\mathcal L_{\Omega,2}^k))$ is defined in (\ref{spaceH}). Then 
\begin{equation}\label{regprop}
\begin{split}
& u\in C_{\partial\Omega}([0,T]\times \Omega)\cap C^{1,2}((0,T)\times \Omega),\quad \text{and}\\
&|\partial_tu(t,x)|\le Ct^{-\gamma}, \text{ for every }(t,x)\in(0,T]\times\Omega, \text{ for some }\gamma\in(0,1),\ C>0,
\end{split}
\end{equation}
where $u$ is defined in (\ref{SRpostRL}). Moreover, $ u$ is the unique classical solution to problem (\ref{postRL}). 
\end{theorem}
\begin{proof} (The notation for constants is consistent with the referenced inequalities.)\\
By Lemma \ref{lem_chen_bounds}-(ii) and Lemma \ref{lem_nvs} we can write our candidate solution (\ref{SRpostRL}) as
\[
u(t,x)=\sum_{n=1}^\infty  E_\beta (-\lambda_n t^\beta)  \langle \phi_0,\psi_n\rangle \psi_n(x)+  \sum_{n=1}^\infty \myF{\langle f(\cdot),\psi_n \rangle}{\lambda_n}(t)\psi_n(x),
\]
and the first series enjoys the regularity properties stated in (\ref{regprop}).  We now prove the same regularity for the second series. Observe that $\sum_{n=1}^\infty \myF{\langle f(\cdot),\psi_n \rangle}{\lambda_n}(t)\psi_n(x) $ converges uniformly to a  function in $C_{\partial\Omega}([0,T]\times \Omega)$, since we have the uniform bound 
\begin{align*}
\sum_{n=1}^\infty 
    | \myF{\langle f(\cdot),\psi_n \rangle}{\lambda_n}(t)\psi_n(x)| &     \leq \sum_{n=1}^\infty c\lambda_n^{-1}\|\langle f(\cdot),\psi_n \rangle\|_{C([0,T])} c_1\lambda_n^{d/(2\alpha)} \\
		 &  \le( cc_1M)\sum_{n=1}^\infty \lambda_n^{-1-k+d/(2\alpha)}\\
		 &  \leq (\tilde c_1c_1cM)\sum_{n=1}^\infty n^{(\alpha /d)(d/(2\alpha)-k-1)}<\infty,
\end{align*}
for any $ k > -1+3d/(2\alpha)$, using the bounds in (\ref{bounds_pair}),  (\ref{first}) and  Lemma \ref{lem_chen_bounds}-(i).	 Further, for $j=1,2$, and for any $x$ in a compact subset $K$ of $\Omega$, the term-wise space derivative of $u$ can be bounded as follows,
\begin{equation}
    \begin{split}
         \sum_{n=1}^\infty |  \myF{\langle f(\cdot),\psi_n \rangle}{\lambda_n}(t) | \|\nabla^j \psi_n \|_\infty 
        &\le  \sum_{n=1}^\infty c\lambda_n^{-1}\|\langle f(\cdot),\psi_n \rangle\|_{C([0,T])}  c_1\lambda_n^{(d+4)/2\alpha}  \\
						&\le(\tilde c_1 c_1cM)\sum_{n=1}^\infty n^{(\alpha/d)\left((d+4)/(2\alpha) - k - 1\right)}<\infty,
    \end{split}
\end{equation}
as
\[ \frac\alpha{d} \left(\frac{d+4}{2\alpha} - k - 1\right) < -1 \iff k > \frac{3d+4-2\alpha }{2\alpha},  \]
 where we use the bounds in (\ref{bounds_pair}),  (\ref{first}) and  Lemma \ref{lem_chen_bounds}-(i).  Thus,  Weierstrass M-test implies that for any $t>0$, $u(t)$ is  a $C^2$ function on every $K\subset\Omega$ compact. For the time regularity we use the inequality  (\ref{second2}) from Lemma \ref{lem_nvs}\footnote{From the proof of Lemma \ref{lem_nvs} it follows that if $\phi_0=f(0)=0$, then  $\partial_t u$  is bounded.}.\\
 By Theorem \ref{thm_wspostRL}, $u$ is also a weak solution to problem (\ref{postRL}), and by Lemma \ref{lem_dualonsmooth} and standard approximation arguments, $u$ satisfies the equalities in (\ref{postRL}). Continuity at $t=0$ can be proved as in Remark \ref{rmk_cont}.

To prove uniqueness, consider two  classical solutions to problem (\ref{postRL}), denoted by $u,v$. Then $w := u-v$ is a classical solution to problem (\ref{postRL}) with $f=0$, $\phi_0=0$. Consider the continuous functions on $[0,T]$,  $t\mapsto\langle w(t),\psi_n\rangle$, $n\in\mathbb N$. If we can justify 
\begin{equation}
D^{\beta}_0\langle w(t),\psi_n\rangle= \langle D^{\beta}_0w(t),\psi_n\rangle=\langle \flap_\Omega w(t),\psi_n\rangle=\langle  w(t),\mathcal L_{\Omega,2}\psi_n\rangle=-\lambda_n \langle w(t),\psi_n\rangle,
\label{eqqq}
\end{equation}
for $t>0$, it follows by \cite[Theorem 6.5 and Theorem 7.2]{kai} that $\langle w(t),\psi_n\rangle= 0$ for every $t\in [0,T],\ n\in\mathbb N$, and we are done. The first equality is a consequence of  $|\partial_r  w(r,y)|\le Cr^{-\gamma}$, for some $\gamma\in(0,1)$. The second and fourth equalities in (\ref{eqqq}) are clear. Now, as $\psi_n\in \text{Dom}(\mathcal L_{\Omega,2}) $, there exists a sequence $\{\psi_{n,j}\}_{j\in\mathbb N}\subset C_c^\infty(\Omega)$, such that  as $j\to\infty$
\begin{equation}
\psi_{n,j}\to\psi_n,\quad \text{and}\quad\flap_\Omega\psi_{n,j} =\mathcal L_{\Omega,2}\psi_{n,j}\to\mathcal L_{\Omega,2}\psi_n,\quad \text{in }L^2(\Omega),
\label{equaa}
\end{equation}
where the  equality in (\ref{equaa}) holds by \cite[Lemma 4.1]{MeerChen}.
Combining  (\ref{equaa}) with the equality (\ref{dual_dfl}) and $\flap_\Omega w(t)\in L^2(\Omega)$ for each $t>0$,  the third equality in (\ref{eqqq}) is proven.
\end{proof}

\section{Stochastic classical solution for problem (\ref{preRL})}\label{sec_sspre}

\subsection{Stochastic representation and continuity at $t=0$}\label{sec_SR}
\begin{lemma}\label{lem_SRs}
Define the function $f_\phi:(0,T]\times\Omega\to\mathbb R$ as  
\begin{equation}
f_\phi(t,x):=\int_t^\infty (\phi(t-r,x)-\phi(t,x))\frac{-\Gamma(-\beta)^{-1}dr}{r^{1+\beta}},
\label{w_p}
\end{equation}
assuming that $\phi\in C_{\infty,\partial \Omega}((-\infty,0]\times \Omega)$,  $\phi(0)\in \text{Dom}(\mathcal L_\Omega)$, and the extension of $\phi$ to $\phi(0)$ on $(0,T]\times\overbar\Omega$ is such that 
\begin{equation}
\phi\in \text{Dom}(\mathcal L_{\beta,\Omega}^\infty),\quad\text{and}\quad\mathcal L_{\beta,\Omega}^\infty\phi= (-D^\beta_{\infty}+\mathcal L_\Omega)\phi.
\label{conditionsonphi}
\end{equation}
 Then $f_\phi\in C([0, T]\times\overbar\Omega)$ and the function $u$ defined in (\ref{SRpostRL}) for $f=f_\phi$ and $\phi_0=\phi(0)$, equals the function $\tilde u$ defined in (\ref{SR}) for $g=0$, on $(0,T]\times \Omega$.
\end{lemma}
\begin{proof} The first claim follows from $f_\phi = -D^\beta_{\infty} \phi\in C([0, T]\times\overbar\Omega)$, using (\ref{conditionsonphi}) and $\mathcal L_\Omega\phi(t,x)=\mathcal L_\Omega\phi(0,x)$ for all $(t,x)\in[0,T]\times \overbar\Omega$. Recall that we write $\tau_{t,x}= \tau_0(t)\wedge\tau_\Omega(x)$. Fix $(t,x)\in (0,T]\times \Omega$. It is enough to justify the following equalities
\begin{align*}
u(t,x)&=\mathbf E\left[\phi(0,X^{x,\alpha}(\tau_0(t))\mathbf 1_{\{\tau_0(t)<\tau_\Omega(x)\}} +\int_0^{\tau_{t,x}} f_\phi\left(-X^{t,\beta}(s),X^{x,\alpha}(s)\right)ds\right]\\
&=\mathbf E\left[\phi(0,x)+\int_0^{\tau_{t,x}} \mathcal L_\Omega  \phi\left(0,X^{x,\alpha}(s)\right)ds +\int_0^{\tau_{t,x}} f_\phi\left(-X^{t,\beta}(s),X^{x,\alpha}(s)\right)ds\right]\\
&=\mathbf E\left[ \int_0^{\tau_{t,x}} \mathcal L_\Omega  \phi\left(-X^{t,\beta}(s),X^{x,\alpha}(s)\right) -D^{\beta}_{\infty} \phi\left(-X^{t,\beta}(s),X^{x,\alpha}(s)\right)ds\right]+\phi(0,x)\\
&=\mathbf E\left[ \int_0^{\tau_{t,x}} \mathcal L_{\beta,\Omega}^\infty \phi\left(-X^{t,\beta}(s),X^{x,\alpha}(s)\right)ds\right]+\phi(0,x)\\
&=\mathbf E\left[  \phi\left(-X^{t,\beta}(\tau_{t,x}),X^{x,\alpha}(\tau_{t,x})\right)\right]\pm\phi(0,x).
\end{align*}
For the second equality we use Dynkin formula with Lemma \ref{thm_semigroups}-(i) and $\phi(0)\in \text{Dom}(\mathcal L_\Omega)$; for the third equality, as we extended $\phi(t,x)=\phi(0,x)$ on $[0,T]\times\Omega$,  we use the identities  $f_\phi(t,x)=-D^\beta_{\infty}\phi(t,x)$ and $\mathcal L_\Omega  \phi(0,x)=\mathcal L_\Omega  \phi(t,x)$ on $(0,T]\times\Omega$;  in the fourth equality we use assumption (\ref{conditionsonphi}); the fifth equality is again an application of Dynkin formula  with Lemma \ref{thm_semigroups}-(iii) and $\phi(t,x)=\phi(0,x)$ on $(0,T]\times  \Omega$. 
\end{proof}

\begin{corollary}\label{cor_SRs}
If $\phi\in C_{b,\partial\Omega}^1((-\infty,0]\times\Omega)$, then for $(t,x)\in(0,T]\times\Omega$
\begin{align}\label{eqdyn}
\mathbf E\left[\phi\big(0,X^{x,\alpha}(\tau_{t,x})\big) +\int_0^{\tau_{t,x}} f_\phi\left(-X^{t,\beta}(s),X^{x,\alpha}(s)\right)ds\right]=\mathbf E\left[  \phi\left(-X^{t,\beta}(\tau_{t,x}),X^{x,\alpha}(\tau_{t,x})\right)\right].
\end{align}

\end{corollary}
\begin{proof} \emph{Step 1.}
We  prove (\ref{eqdyn}) for $\phi\in C_{\infty,\partial\Omega}^1((-\infty,0]\times\Omega)\cap \{\partial_t f(0)=0\}$ with compact support in $(-\infty,0]\times\overbar\Omega$. For such $\phi$, let $K>0$ such that $\phi$ is supported in $(-K,0]\times\overbar\Omega$. By the same arguments as in the proof of Lemma \ref{thm_semigroups}-(ii), it follows that $\text{Span}\{C([-K,0])\cap\{f(-K)=f(0)=0\}\cdot C_{\partial\Omega}(\Omega)\}$ is dense in $C_{\partial\Omega}([-K,0]\times\Omega)\cap\{f(-K)=f(0)=0\}$ with respect to the supremum norm. We can use this fact to construct a sequence $\{\phi_n\}_{n\in\mathbb N}\in \text{Span}\{C_{\infty}^1(-\infty,0])\cap\{f'(0)=0\}\cdot C_{\partial\Omega}(\Omega)\}$ such that 
\[
\|\phi_n-\phi\|_{C((-\infty,0]\times\overbar\Omega)}+\|\partial_t(\phi_n-\phi)\|_{C((-\infty,0]\times\overbar\Omega)}\to 0,\quad\text{as }n\to\infty.
\]
Moreover, it follows that $f_{\phi_n}\to f_\phi$  as $n\to\infty$ pointwise on $[0,T]\times\Omega$ and $\sup_n\|f_{\phi_n}\|_{C([0,T]\times\overbar\Omega)}$ is finite. It remains to show that (\ref{eqdyn}) holds for functions in $\text{Span}\{C_{\infty}^1(-\infty,0])\cap\{f'(0)=0\}\cdot C_{\partial\Omega}(\Omega)\}$, as  DCT applied to the sequences above yields the claim. By Lemma \ref{thm_semigroups}-(iii) with $C_\beta^\infty=C_\infty^1((-\infty,T])$, Proposition \ref{prop_csg} and Lemma \ref{lem_SRs}, equality (\ref{eqdyn})  holds for $\phi\in\text{Span}\{C_{\infty}^1((-\infty,0])\cap\{f'(0)=0\}\cdot\text{Dom}(\mathcal L_\Omega))\} $. As  $\text{Dom}(\mathcal L_\Omega)$ is dense in $C_{\partial\Omega}(\Omega)$, equality (\ref{eqdyn}) holds for    $\phi\in \text{Span}\{C_{\infty}^1((-\infty,0])\cap\{f'(0)=0\}\cdot C_{\partial\Omega}(\Omega)\} $ by DCT. \\
\emph{Step 2.} For  $\phi\in C^1_{b,\partial\Omega}((-\infty,0]\times\Omega)$,  take a sequence $\{\phi_n\}_{n\in\mathbb N}\subset C_{\infty,\partial\Omega}^1((-\infty,0]\times\Omega)\cap \{\partial_t f(0)=0\}$  compactly supported in $(-\infty,0]\times\overbar\Omega$, such that $\phi_n\to\phi$  pointwise on $(-\infty,0]\times\Omega$, and  $\sup_n\|\phi_n\|_{C((-\infty,0]\times\overbar\Omega)}+\sup_n\|\partial_t\phi_n\|_{C((-\infty,0]\times\overbar\Omega)}<\infty$. Then $f_{\phi_n}\to f_\phi$  pointwise on $[0,T]\times\Omega$ and $\sup_n\|f_{\phi_n}\|_{C([0,T]\times\overbar\Omega)}<\infty$. Finally, apply DCT to both sides of (\ref{eqdyn}).
\end{proof}

\begin{remark}\label{rmk_cont}
If we can apply Corollary \ref{cor_SRs}, then we can prove continuity at $t=0$ for the solution (\ref{SR}) via the following argument 
\begin{align*}
|\text{Formula (\ref{SRpostRL})}-\phi_0(x)|&\le|\mathbf E\left[\phi_0\left(X^{x,\alpha}(\tau_0(t)\wedge \tau_\Omega(x))\right)-\phi_0(x)\right]|+\|f\|_\infty \mathbf E\left[\tau_0(t)\right]\\
&= o_{t\downarrow 0}(1)\ +\|f\|_\infty\frac{t^\beta}{\Gamma(\beta+1)},
\end{align*}
for each $x\in\Omega$, using stochastic continuity of the process\footnote{This follows as $X^{x,\alpha}(s)$ is right continuous and $\tau_0(t)$ is right continuous, non-decreasing with $\tau_0(0)=0$.} $t\mapsto X^{x,\alpha}(\tau_0(t))$ at $t=0$. One could also use stochastic continuity at $t=0$ of $-X^{t,\beta}(\tau_0(t))=t-X^{\beta}(\tau_0(t))$, bypassing  Corollary \ref{cor_SRs}. In Proposition \ref{thm_ctsat0} in the Appendix we prove continuity at $t=0$ by proving a bound on big overshootings $-X^{t,\beta}(\tau_0(t))$ for small times.
\end{remark}

\subsection{Equivalence of the classical solutions to problems (\ref{preRL}) and (\ref{postRL})}\label{sec_eoss}

\begin{definition} Let $\phi \in C_{b,\partial\Omega}((-\infty,0]\times\Omega)$ and $g\in C((0,T]\times\Omega)$. A function $\tilde u\in C_{b,\partial\Omega}((-\infty,T]\times \Omega)\cap C^{1,2}((0,T)\times \Omega)$  such that $|\partial_t\tilde u(t,x)|\le Ct^{-\gamma}, \text{ for every }(t,x)\in(0,T]\times\Omega, \text{ for some }\gamma\in(0,1),\ C>0$, is said to be a \emph{classical solution to problem} (\ref{preRL}) if $\tilde u$ satisfies the identities  in (\ref{preRL}), and for every $x\in \Omega$
$$
\lim_{t\downarrow 0}|\tilde u(t,x)-\phi(0,x)|=0.
$$
\end{definition}
\begin{lemma}\label{lem_equivsol}
Let  $\phi\in C_{b,\partial\Omega}((-\infty,0]\times\Omega)$ such that $f_\phi\in C((0,T]\times\Omega)$, where  $f_\phi$  is defined in (\ref{w_p}), and let $g \in C((0,T]\times\Omega)$. Then, if  $u$ is a classical solution to problem (\ref{postRL}) with $f=f_\phi+g$ and $\phi_0=\phi(0)$, then the extension 
\begin{equation*}
\tilde u:=\left\{\begin{split}
u,\quad\text{in }& (0,T]\times\overbar\Omega,\\
\phi,\quad\text{in }& (-\infty,0]\times\Omega,
\end{split}
\right.
\label{ext}
\end{equation*}
 is a classical solution to problem (\ref{preRL}). Conversely, if  $\tilde u$ is a classical solution to problem (\ref{preRL}), then the restriction of $\tilde u$ to $[0,T]\times \overbar\Omega$ is a classical solution to problem (\ref{postRL}) with   $f=f_\phi+g$ and  $\phi_0=\phi(0)$.
\end{lemma}
\begin{proof}
The equivalence of convergence to initial data and the required regularities are clear. It is also immediate that $\flap_\Omega  u = \flap_\Omega  \tilde u$ on $(0,T]\times\Omega$. Write $\nu(r)=-\Gamma(-\beta)^{-1}r^{-1-\beta}$. On $(0,T]\times\Omega$ we have the equality
\begin{align*}
-D_{\infty}^{\beta}   \tilde u(t,x)=&\ \int_0^{\infty}(\tilde u(t-r,x)-\tilde u(t,x))\,\nu(r)dr\\
=&\ \int_0^{t}(\tilde u(t-r,x)-\tilde u(t,x))\,\nu(r)dr+\int_t^{\infty}\phi(t-r,x)\,\nu(r)dr\\
&\quad-\tilde u(t,x))\int_t^\infty\nu(r)dr \pm \phi(0,x)\int_t^\infty\,\nu(r)dr\\
=&\ -D_{0}^{\beta}   \tilde u(t,x)+f_\phi(t,x).
\end{align*}
This is enough to prove both directions.
\end{proof}

\subsection{Main result}

\begin{theorem}\label{thm_main} Let $\Omega\subset \mathbb R^d$ be a regular set. Assume that $\phi\in C_{b,\partial\Omega}^1((-\infty,0]\times\Omega)$ with $\phi(0)\in \text{Dom}(\mathcal L_{\Omega,2}^k)$ and $f_\phi,g\in C^1([0,T];\text{Dom}(\mathcal L_{\Omega,2}^k))$,  for some $k>-1+(3d+4)/(2\alpha)$, where $f_\phi$ is defined in (\ref{w_p}) and  $C^1([0,T];\text{Dom}(\mathcal L_{\Omega,2}^k))$ is defined in (\ref{spaceH}). Then 
\begin{align*}
&\tilde u\in C_{b,\partial\Omega}((-\infty,T]\times \Omega)\cap C^{1,2}((0,T)\times \Omega),\quad \text{and}\\
&|\partial_t\tilde u(t,x)|\le Ct^{-\gamma}, \text{ for every }(t,x)\in(0,T]\times\Omega, \text{ for some }\gamma\in(0,1),\ C>0,
\end{align*}
where $\tilde u$ is defined as in (\ref{SR}).
Moreover, $\tilde u$ is the unique classical solution to problem (\ref{preRL}). 
\end{theorem}
\begin{proof}
By the assumptions on $\phi$ and $g$, and Lemma \ref{lem_equivsol}, existence and uniqueness of classical solutions follows by  Theorem \ref{thm_sspostRL} with $\phi_0=\phi(0)$ and $f=f_\phi+g$. Now apply Corollary \ref{cor_SRs} to obtain the stochastic representation (\ref{SR}) from the stochastic representation (\ref{SRpostRL}).
\end{proof}
\begin{remark}\label{rmk_HK}
Using Corollary \ref{cor_SRs} (or \cite[Theorem 1 for $\lambda=0$]{IkeWa62}), $\mathbf P[-X^{t}(\tau_0(t))\in\{0\}]=0$ for every $t>0$ (see \cite[III, Theorem 4]{ber}) and the independence of $X^{x,\alpha}$ and $-X^{t,\beta}$, one can show that for $(t,x)\in(0,T]\times\Omega$
\begin{equation*}
\mathbf E\left[\phi\left(-X^{t,\beta}(\tau_0(t)),X^{x,\alpha}(\tau_0(t))\right)\mathbf 1_{\{\tau_0(t)<\tau_\Omega(x)\}}\right]=\int_{-\infty}^0\int_\Omega \phi(r,y)H_{\beta,\alpha}^{t,x}(r,y)\,dr\,dy,
\end{equation*}
where
\begin{equation*}\begin{split}
H_{\beta,\alpha}^{t,x}(r,y)&=\int_0^t\frac{-\Gamma(-\beta)^{-1}}{(z-r)^{1+\beta}}\left(\int_0^\infty  p^\Omega_s(x,y)p^\beta_s(t-z)\,ds\right)\,dz.
\end{split}
\end{equation*}
It is straightforward to compute for $(t,x)\in(0,T]\times\Omega$
\begin{equation*}
\mathbf E\left[\int_0^{\tau_0(t)\wedge\tau_\Omega(x)} g\left(-X^{t,\beta}(s),X^{x,\alpha}(s)\right)ds\right]=  \int_0^t\int_\Omega g(z,y)\left(\int_0^\infty p^\Omega_s(x,y)p^\beta_s(t-z)\,ds\right)\,dz\,dy.
\end{equation*}
\end{remark}

\begin{remark}
Notice that the value $\phi(0)$ does not contribute to the solution (\ref{SR}) because $\mathbf P[-X^{t}(\tau_0(t))\in\{0\}]=0$ for all $t>0$. However, $u(t)\to \phi(0)$ as $t\downarrow 0$. We discuss the continuity of the solution at $t=0$ in more detail in Appendix \ref{app_cont}.  
\end{remark}
\begin{remark}
We could drop the condition $\|\partial_t\phi\|_\infty<\infty$ in  Theorem \ref{thm_main}, by weakening Corollary \ref{cor_SRs}, for example to $\phi$ being $\beta^*$-H\"older continuous at $t=0$, for some $ \beta^*>\beta$ and $\phi\in L^\infty((-\infty,0)\times\Omega)$. This is essentially because  $\lim_{t\downarrow 0}f_\phi(t)$ remains well-defined.  However, in order to apply  Theorem \ref{thm_sspostRL} in the proof of Theorem \ref{thm_main} we need to assume $f_\phi\in C^1([0,T];\text{Dom}(\mathcal L_{\Omega,2}^k))$. Hence,   a minimal requirement is that $\phi$ is continuously differentiable in time and both  $\phi$ and $\partial_t\phi$ are $\mathcal O(|r|^{\beta_*})$ at $-\infty$ and $\beta^*$-H\"older continuous at $0$, for some $\beta_*<\beta< \beta^*$, as we need $f_\phi$ and $\partial_tf_\phi$ to be continuous on $[0,T]\times\overline\Omega$.
\end{remark}

\begin{remark}
Suppose that $\phi\in C_{\infty,\partial\Omega}^{2,2 k}((-\infty,0]\times\Omega)$ and  $\phi(t)$ along with its  partial derivatives in space are compactly supported in $\Omega$, for each $t\in (-\infty,0]$, where $k\in\mathbb N$ and  $k>-1+(3d+4)/(2\alpha)$. Then, an application of Remark \ref{rmk_HTk} implies that $f_\phi\in C^1([0,T];\text{Dom}(\mathcal L_{\Omega,2}^k))$.
\end{remark}

\section{Intuition for the stochastic solution (\ref{SR})}\label{sec_int}
We discuss the intuition for the stochastic representation (\ref{SR}) as the solution to the EE (\ref{preRL}). Let us write  $-W(t)=t - X^{\beta}(\tau_0 (t))=-X^{t,\beta}(\tau_0 (t))$. Then  $W (t)$ is the overshoot of the subordinator $ X^{\beta}$ with respect to the barrier $t$, recalling that   the first exit time/inverse subordinator is given by $\tau_0 (t)=\inf\{s>0:\ t\le X^{\beta}(s)\}$. To ease notation we write $ Y^x:= \{X^{x,\alpha}(\tau_0(t))\mathbf 1_{\{\tau_0(t)<\tau_\Omega(x)\}}\}_{t\ge 0}$. Let us start from the intuition of Caputo  EEs, as if $\phi(t,x)=\phi(0,x)=:\phi_0(x)$ for every $t\in(-\infty,0]\times\Omega$, then the solution (\ref{SR}) reads
\begin{equation}
u(t,x)=\mathbf E\left[\phi_0(Y^{x}(t))\right],
\label{intro_SRC}
\end{equation}
and the EE  (\ref{preRL}) equals the Caputo EE (\ref{postRL}) (for $g=f=0$). The probabilistic object  defining the solution (\ref{intro_SRC}) is the anomalous diffusion $Y^x$. Recall that the particle $Y^x$ is either trapped or diffusing.\\
 \textbf{Key observation}: reasoning path-wise, for some $\bar x\in \Omega$
\[
 \text{the interval $(t_1,t_2)$ is the maximal open interval so that $t\mapsto Y^x(t)=\bar x$ is constant}
\] 
$$\iff$$
\[
 \text{the interval $(t_1,t_2)$ is the maximal open interval so that $t\mapsto\tau_0(t)$ is constant}
\] 
$$\iff$$
\[
 \text{the interval $(t_1,t_2)$ is the maximal open interval so that $t\mapsto X^{\beta}(\tau_0(t))$ is constant}
\]
$$\iff$$
\[
 \text{$X^{\beta}(\tau_0(t)-)=t_1$ and $X^{\beta}(\tau_0(t))=t_2$, (i.e. $X^\beta$ jumped from $t_1$ to $t_2$).}
\] 
The last statement implies that 
\[
 \text{ $W (t)= X^{\beta}(\tau_0(t))-t=t_2-t\in (0,t_2-t_1)$ for every $t\in (t_1,t_2)$},
\]
which is the trapping/waiting time of $Y^x(t)$.  In words: the event of the diffusion $ Y^x$  being trapped at a point $\bar x\in\Omega$ at time $t$ until time $t+s$  happens precisely when $W (t)=s$. Hence the law of  $-W (t)$ provides a weighting of the initial condition $\phi(\bar x)$ depending on the trapping/waiting time of $Y^x(t)$. Notice that the process $t\mapsto -W(t)$ is self-similar with index $1$ and it is composed by right continuous $45$ degrees increasing slopes with $0$ leftmost limit  (see Figure $1$).
\begin{figure}[h!] \label{fig1}
\centering
\includegraphics[trim = 0cm 0cm 0cm .2cm, clip=true,height=6cm]{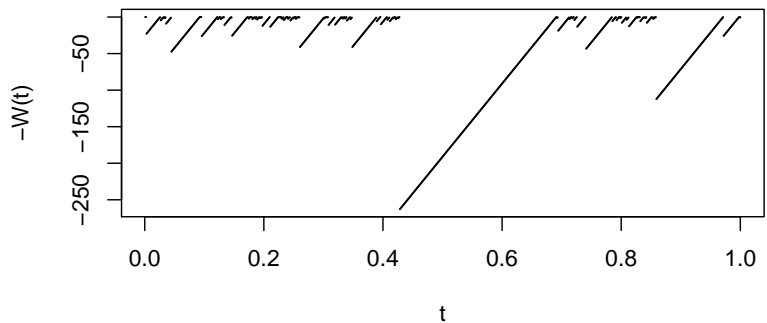}
 \caption{A typical path of the overshoot $t\mapsto -W(t)=-X^{t,\beta}(\tau_0(t))$, $\beta=0.9$. }
\end{figure}

\subsection{A non-memory interpretation}
It is possibly appealing to think about the values $(-\infty,0)\times \Omega$ for the initial condition $\phi$ as the `depth' underneath the surface $\{0\}\times \Omega$  where the particle $Y^x$ moves. Then one can think about the particle $Y^x(t)$ as falling instantaneously at the bottom of a hole/trap of depth $|t_2-t_1|$, and then taking time $|t_2-t_1|$ to climb back up to the surface. Then, at time $t$ one can observe the particle being $|t_2-t|$-depth-units down in the hole. From this viewpoint, once the particle is in the hole it just drifts upward with unit speed. As a quick example, consider the variable separable initial condition $\phi(t,x)= p(t)q(x)$ where $p(t)=\mathbf 1_{\{t< -1\}}$. Then the solution reads for $t>0$
\begin{align*}
u(t,x)&=\mathbf E\left[ q(Y^{x}(t))\mathbf 1_{\{ W(t)>1\}}\right]\\
&=\mathbf E\left[ q(Y^{x}(t)) |  Y^x(t)\text{ is more than 1 unit deep in a trap}\right] \\
\Big(&=\mathbf E\left[ q(Y^{x}(t)) |  Y^x(t)\text{ is trapped for more than 1 time-unit}\right]\Big).
\end{align*} 
Hence, in this example the diffusive particle $Y^x$ will have to be a least a unit deep in a hole (trapped for at least a unit time) for the values at its trapping point at its depth (in the past) to contribute to the solution.

\renewcommand{\thesection}{A} 
\renewcommand{\thesubsection}{\thesection.\Roman{subsection}}

\section{Appendix}

\subsection{Continuity of solution (\ref{SR}) at $t=0$}\label{app_cont}

\begin{proposition}\label{prop:bound}
For every $p,\varepsilon>0$, the following bound on  small overshootings holds, 
\begin{equation*}
\mathbf P[X^{t,\beta}(\tau_0(t))\le \varepsilon]\ge(1-p), \quad \text{for every }  t\le \varepsilon p^{\frac{1}{\beta}}.
\end{equation*}
\end{proposition}
\begin{proof}
With the first equality holding by \cite[Theorem 1 for $\lambda=0$]{IkeWa62} along with the identity (\ref{ident3}), compute 
\begin{align*}
\mathbf P[X^{t,\beta}(\tau_0(t))\le \varepsilon]&=\int_{-\varepsilon}^0 \left(\frac{1}{\Gamma(\beta)}\int_{0}^{ t} (-\partial_y(y-r)^{-\beta})\frac{ (t-y)^{\beta-1}}{\Gamma(1-\beta)}\, d y\right)\,dr\\
&=\int_{-\varepsilon}^0 \left(\frac{\beta}{\Gamma(\beta)\Gamma(1-\beta)}\int_{0}^{ t} (y-r)^{-\beta-1} (t-y)^{\beta-1}\, d y\right)\,dr\\
&=\frac{-\Gamma(\beta)^{-1}}{\Gamma(-\beta)}\int_{0}^{t} (t-y)^{\beta-1} \left( \int_{-\varepsilon}^0(y-r)^{-\beta-1}  d r\right)\,dy\\
&=\frac{-\Gamma(\beta)^{-1}\beta^{-1}}{\Gamma(-\beta)}(a-a_\varepsilon(t)),
\end{align*}
where $a_\varepsilon(t):= \int_{0}^{t} (t-y)^{\beta-1}(y+\varepsilon)^{-\beta}\,dy$  and $a:=\int_{0}^{t} (t-y)^{\beta-1}y^{-\beta}\,dy=\Gamma(\beta)\Gamma(1-\beta)$ for every $t>0$. Now pick $\tilde t=\varepsilon p^{1/\beta}$. Then for every  $0\le y\le\tilde t$
\begin{equation*}
(y+\varepsilon)^{-\beta}= (y+p^{-1/\beta}\tilde t)^{-\beta}\le p\tilde t^{-\beta}\le py^{-\beta},
\end{equation*}
hence for every $t\le \tilde t$
\[
\frac{a_\varepsilon(t)}{a}=\frac{\int_{0}^{t} (t-y)^{\beta-1}(y+\varepsilon)^{-\beta}dy}{\int_{0}^{t} (t-y)^{\beta-1}y^{-\beta}\,dy}\le p.
\]
  Then $a_\varepsilon(t)\le p a$ for every $t\le \tilde t$, which is equivalent to $a-a_\varepsilon(t)\ge (1-p)a$ for every $t\le \tilde t$. And so we obtain
\[
\mathbf P[X^{t,\beta}(\tau_0(t))\le \varepsilon]\ge(1-p) \frac{-\Gamma(\beta)^{-1}}{\Gamma(-\beta)}\beta^{-1}\Gamma(\beta)\Gamma(1-\beta)=(1-p).
\]
\end{proof}

We now use the bound in Proposition \ref{prop:bound} to prove the following continuity result
\begin{proposition}\label{thm_ctsat0}
Consider the function $\tilde u$ defined in (\ref{SR}), with an arbitrary $\Omega$-valued stochastic (sub-)process $X^x$ in place of $X^{x,\alpha}$, such that $t\mapsto X^{x}(\tau_0(t))$ is stochastically continuous at $t=0$. Also assume $\phi\in B((-\infty,0]\times\Omega))$ and $\phi$ is continuous  at every point in $\{0\}\times\Omega$. Then  for every $x\in\Omega$
\[
\lim_{t\downarrow 0}|\tilde u(t,x)- \phi(0,x)|=0.
\]
\end{proposition}

\begin{proof}
 Let $x\in\Omega$. Let $\delta>0$ be arbitrary. Pick $\varepsilon,\varepsilon'>0$ such that 
\[
\sup_{(s,y)\in(-\varepsilon,0]\times B_{\varepsilon '}(x)}|\phi(s,y)-\phi(0,x)|\le \delta.
\]
Then 
\begin{align*}
|\tilde u(t,x)-\phi(0,x)|\le&\ \left|\mathbf E\left[(\phi(-X^{t,\beta}(\tau_0(t)),X^{x}(\tau_0(t))-\phi(0,x))\mathbf 1_{\{X^{t,\beta}(\tau_0(t))>\varepsilon\}}\right]\right|\\
&\ +\left|\mathbf E\left[(\phi(-X^{t,\beta}(\tau_0(t)),X^{x}(\tau_0(t)))-\phi(0,x))\mathbf1_{\{X^{t,\beta}(\tau_0(t))\le\varepsilon\}}\right]\right|\\
\le&\ 2\|\phi\|_{\infty}\mathbf P[X^{t,\beta}(\tau_0(t))>\varepsilon]\\
&\ +\mathbf E\left[|\phi(-X^{t,\beta}(\tau_0(t)),X^{x}(\tau_0(t)))-\phi(0,x)|\mathbf1_{\{X^{t,\beta}(\tau_0(t))\le\varepsilon, |X^{x}(\tau_0(t)))-x|\le\varepsilon'\}}\right]\\
&\ +\mathbf E\left[|\phi(-X^{t,\beta}(\tau_0(t)),X^{x}(\tau_0(t)))-\phi(0,x)|\mathbf1_{\{X^{t,\beta}(\tau_0(t))\le\varepsilon,|X^{x}(\tau_0(t)))-x|>\varepsilon'\}}\right]\\
\le&\ 2\|\phi\|_{\infty}\mathbf P[X^{t,\beta}(\tau_0(t))>\varepsilon] +\delta +2\|\phi\|_{\infty}\mathbf P[|X^{x}(\tau_0(t)))-x|>\varepsilon']
\end{align*}
 Now, by Proposition \ref{prop:bound}, for all $ t\le \delta^{\frac 1 \beta}\varepsilon$ it holds that $\mathbf P[X^{t,\beta}(\tau_0(t))>\varepsilon]\le \delta$.  Then  the estimate above reads
\begin{align*}
|\tilde u(t,x)-\phi(0,x)|\le&\ 2\|\phi\|_{\infty}\delta+\delta+ 2\|\phi\|_{\infty}\mathbf P[|X^{x}(\tau_0(t)))-x|>\varepsilon'],\quad \text{for every }t\le \delta^{\frac 1 \beta}\varepsilon.
\end{align*}
To conclude, by stochastic continuity,  pick a possibly smaller threshold $\bar t$  to obtain 
\[
\mathbf P[|X^{x}(\tau_0(t)))-x|>\varepsilon']\le \delta\quad \text{for every }t\le \bar t.
\]
\end{proof}
\begin{remark}
The continuity at $t=0$ of Proposition \ref{thm_ctsat0}  is not obvious. For example it is clear that Proposition \ref{thm_ctsat0} fails if we replace $-X^{t,\beta}$ with a decreasing Poisson process. In fact Proposition \ref{thm_ctsat0} fails in general if we replace  $-X^{t,\beta}$ with a decreasing compound Poisson process  $-N^t(s)$  with generator
$$ 
-D _\infty^{(\nu)} f(t):=\int_0^\infty (f(t-r)-f(t))\,\nu(dr),\quad \text{where}\quad0<\lambda:=\int_0^\infty \nu(dr)<\infty. 
$$ 
To see this, observe that for every $\varepsilon,t>0$
\begin{align*}
\mathbf P\left[N^t\left(\tau_0 (t)\right)> \varepsilon\right] \ge \mathbf P\left[\text{first jump of $N^t$ is  greater than $t+\varepsilon$}\right]=\int_{t+\varepsilon}^\infty \frac{\nu(dr)}{\lambda},
\end{align*} 
and note that the right hand side is non-decreasing as $t\downarrow 0$, where $\tau_0 $ is the left continuous inverse of $N^0$. As $\int_0^\infty\nu(dr)>0 $ we can choose $\varepsilon_0>0$ and $\bar t>0$  so that 
$$
\inf_{t\le\bar t}\mathbf P\left[N^t\left(\tau_0 (t)\right)> \varepsilon_0\right]\ge\int_{\bar t+\varepsilon_0}^\infty\frac{\nu(dr)}{\lambda}=:c>0.
$$
 Now, consider a continuous non-negative  $\phi$ with $\phi(0)=0$,   such that $\inf_{r\in (-\infty,-\varepsilon_0]}\phi(r)>0$. Then for every $t\le \bar t$
\begin{align*}
|\tilde u(t)-\phi(0)|&=\mathbf E\left[\phi(-N^t\left(\tau_0(t)\right)\left(\mathbf1_{\{N^t(\tau_0(t))> \varepsilon_0\}}+\mathbf1_{\{N^t(\tau_0(t))\le \varepsilon_0\}}\right)\right]\\
&\ge \mathbf E\left[\phi(-N^t\left(\tau_0(t)\right))\mathbf1_{\{N^t(\tau_0(t))> \varepsilon_0\}}\right]\\
&\ge \inf_{r\in (-\infty,-\varepsilon_0]}\phi(r)\mathbf P\left[N^t\left(\tau_0(t)\right)> \varepsilon_0\right]\\
&\ge \inf_{r\in (-\infty,-\varepsilon_0]}\phi(r)c>0.
\end{align*}
\end{remark}

\subsection{Proof of Lemma \ref{thm_semigroups}-(i)-(ii)-(iii)}\label{proof-i-ii-iii}
The three proofs are essentially the same,  hence  we prove only (ii). 

Note that $P^{\beta,\text{kill}}_sP^\Omega_r=P^\Omega_rP^{\beta,\text{kill}}_s$ for every $s,r\ge 0$, and that 
\[
\|P^\Omega_s f\|_{C([0,T]\times\overbar\Omega)},\ \|P^{\beta,\text{kill}}_s f\|_{C([0,T]\times\overbar\Omega)}\le \|f\|_{C([0,T]\times\overbar\Omega)},
\] for every $f\in C_{0,\partial\Omega}([0,T]\times\Omega)$, $s\ge0$. It is then easy to prove that $P^{\beta,\Omega,\text{kill}}$ is sub-Feller semigruop on $ C_{0,\partial\Omega}([0,T]\times\Omega)$. We denote the generator of $P^{\beta,\Omega,\text{kill}}$ by $(\mathcal L_{\beta,\Omega}^{\text{kill}},\text{Dom}(\mathcal L_{\beta,\Omega}^{\text{kill}}))$. Let $f=pq $, where $p\in \mathcal C_\beta^{\text{kill}}$ and $q\in\mathcal C_\Omega $. Then, by a standard  triangle inequality argument, we obtain
\begin{align*}
\Big|\frac{P^{\beta,\text{kill}}_hP^\Omega_h f(t,x)-f(t,x)}{h}&-(\mathcal L_\beta^{\text{kill}} +\mathcal L_\Omega )f(t,x)\Big|\\
\le&\ \|p\|_{C([0,T])}\left\|\frac{P^\Omega_hq-q}{h} -\mathcal L_\Omega q\right\|_{C(\overbar\Omega)} + \|\mathcal L_\Omega q\|_{C(\overbar\Omega)}\left\| P^{\beta,\text{kill}}_h p-p\right\|_{C([0,T])} \\
&\ +\|q\|_{C(\overbar\Omega)}\left\| \frac{P^{\beta,\text{kill}}_h p-p}{h}-\mathcal L_\beta^{\text{kill}}p\right\|_{C([0,T])}\to 0,
\end{align*} 
as $h\downarrow 0$. An induction argument proves that $ \text{Span} \{\mathcal C_\beta^{\text{kill}}\cdot \mathcal C_\Omega \}\subset  \text{Dom}(\mathcal L_{\beta,\Omega}^{\text{kill}})$ and $\mathcal L_{\beta,\Omega}^{\text{kill}}= (\mathcal L_\beta^{\text{kill}} +\mathcal L_\Omega)$ on $\text{Span} \{\mathcal C_\beta^{\text{kill}}\cdot \mathcal C_\Omega \}$. Observing that  $\text{Span} \{\mathcal C_\beta^{\text{kill}}\cdot \mathcal C_\Omega \}$ is invariant under $P^{\beta,\Omega\text{kill}}$ and it is a subspace of $\text{Dom}(\mathcal L_{\beta,\Omega}^{\text{kill}})$, if we can prove that $\text{Span} \{\mathcal C_\beta^{\text{kill}}\cdot \mathcal C_\Omega\}$ is dense in $ C_{0,\partial\Omega}([0,T]\times\Omega)$, we are done by  \cite[Lemma 1.34]{Schilling}. So proceed by noting that  set $\text{Span} \{C^\infty([0 ,T])\cdot C^\infty(\overbar\Omega)\}$ is a sub-algebra of $C([0,T]\times\overbar\Omega)$ that contains constant functions and separates points. Hence  $\text{Span} \{C^\infty([0 ,T])\cdot C^\infty(\overbar\Omega)\}$ is dense in $C([0,T]\times\overbar\Omega)$ by Stone-Weierstrass Theorem for compact\footnote{In the case of unbounded domains (part (iii) of the current lemma) use the Stone-Weierstrass Theorem for locally compact Hausdorff spaces.} Hausdorff spaces. We now prove density of the following set
\begin{align*}
\text{Span} \{C_c^\infty((0 ,T])\cdot C_c^\infty(\Omega)\}&\subset C_{0,\partial\Omega}([0,T]\times\Omega).
\end{align*}
For   $f\in  C_{0,\partial\Omega}([0,T]\times\Omega)$ we take a sequence $\{f_n\}_{n\in\mathbb N}\subset \text{Span} \{C^\infty([0 ,T])\cdot C^\infty(\overbar\Omega)\} $ such that $f_n\to f$, where $f_n(t,x)=\sum_{i=1}^{N_n}p_{i,n}(t)q_{i,n}(x)$, for some $N_n\in\mathbb N$ depending on $n\in\mathbb N$. Let $ 1_{T,n}\in C_c^\infty((0,T])$ and $ 1_{\Omega,n}\in C_c^\infty(\Omega)$ be smooth functions for each $n\in\mathbb N$, such  that $ 0\le 1_{T,n}, 1_{\Omega,n}\le 1 $, $  1_{T,n}(t)=  1_{\Omega,n}(x)=1$ for  $t\in (\frac1n,T]$ and $x\in K_n$, and $  1_{T,n}(t)=  1_{\Omega,n}(x)=0$ for  $t\in (0,\frac1{n+1}]$ and $x\in\Omega\backslash K_{n+1}$,  where $K_n$ is compact, $K_n\subset K_{n+1}\subset \Omega$ for each $n$, and $\cup_n K_n=\Omega$.  Define for each $n\in\mathbb N$, 
$$(t,x)\mapsto\tilde f_n(t,x):=\sum_{i=1}^{N_n}p_{i,n}(t)1_{T,n}(t)q_{i,n}(x)1_{\Omega,n}(x)\in \text{Span} \{C_c^\infty((0 ,T])\cdot C_c^\infty(\Omega)\}.$$
  Then, as $n\to\infty$
\begin{align*}
\|\tilde f_n -f\|_{C ([0,T]\times \Omega)}\le&\ \|f_n -f\|_{C ([\frac1n,T]\times K_n)} + \|\tilde f_n -f\|_{C \left((\frac{1}{n+1},\frac1n]\times \overbar\Omega\cup [0,T]\times K_{n+1}\backslash K_n\right)}\\
&\ +\|f\|_{C \left([0,T]\times \overbar\Omega \backslash K_{n+1}\cup [0,\frac{1}{n+1}]\times \overbar\Omega\right)}\to 0.
\end{align*}
As $C_c^\infty(\Omega) \not\subset \text{Dom}(\mathcal L_\Omega)$ we need to work a bit more. For any  $u\in C_{0,\partial\Omega}([0,T]\times\Omega) $ we can now take a uniformly approximating sequence $\{u_n\}_{n\in\mathbb N}\subset \text{Span} \{C_c^\infty((0 ,T])\cdot C_c^\infty(\Omega)\}$. Denote $u_n(t,x)=\sum_{i=1}^{N_n}p_{i,n}(t)q_{i,n}(x)$, for some $N_n\in\mathbb N$ depending on $n\in\mathbb N$, where $p_{i,n}\in C_c^\infty((0 ,T]),\ q_{i,n}\in C_c^\infty(\Omega)$ are non-zero, for each $i\in \{1,...,N_n\}$, $n\in\mathbb N$. As $\mathcal C_\beta$ and $\mathcal C_\Omega$ are dense in $C_0([0 ,T])\supset C_c^\infty((0 ,T])$ and $ C_{\partial\Omega}(\Omega)\supset C_c^\infty(\Omega)$, respectively, we can pick $\{(\tilde p_{i,n}, \tilde q_{i,n}):i\in  \{1,...,N_n\},n\in\mathbb N\}\subset \mathcal C_\beta\times \mathcal C_\Omega$, in the following fashion: for each triplet  $(N_n, p_{i,n},  q_{i,n})$, first pick $\tilde p_{i,n}$ so that 
\[
\| p_{i,n}-\tilde p_{i,n}\|_{C[0,T]}\le\frac{1}{ nN_n\|q_{i,n}\|_{C[0,T]}},
\]
secondly pick $\tilde q_{i,n}$ so that 
\[
\| q_{i,n}-\tilde q_{i,n}\|_{C[0,T]}\le \frac{1}{nN_n\|\tilde p_{i,n}\|_{C[0,T]}}.
\]
 Then, after defining $\tilde u_n(t,x):=\sum_{i=1}^{N_n}\tilde p_{i,n}(t)\tilde q_{i,n}(x)$, we obtain
\begin{align*}
\|u-\tilde u_n\|_\infty&\le \|u- u_n\|_\infty+\|u_n-\tilde u_n\|_\infty\\
&\le \|u- u_n\|_\infty+\sum_{i=1}^{N_n}\| p_{i,n} q_{i,n}-\tilde  p_{i,n}\tilde q_{i,n}\|_\infty\\
&\le \|u- u_n\|_\infty+\sum_{i=1}^{N_n}\left(\| q_{i,n}\|_\infty\|p_{i,n}-\tilde  p_{i,n}\|_\infty+\| \tilde p_{i,n}\|_\infty\|q_{i,n}-\tilde  q_{i,n}\|_\infty\right)\\
&\le \|u- u_n\|_\infty+\sum_{i=1}^{N_n}\left(\frac{\| q_{i,n}\|_\infty}{ nN_n\|q_{i,n}\|_{C[0,T]}}+\frac{\| \tilde p_{i,n}\|_\infty}{nN_n\|\tilde p_{i,n}\|_{C[0,T]}}\right)\\
&= \|u- u_n\|_\infty+\sum_{i=1}^{N_n}\frac{2}{nN_n}\\
&\le \|u- u_n\|_\infty+\frac{2}{n}\to 0,\quad \quad \text{as }n\to\infty. 
\end{align*}

\let\oldbibliography\thebibliography
\renewcommand{\thebibliography}[1]{\oldbibliography{#1}
\setlength{\itemsep}{0pt}}

\end{document}